\documentclass[12pt]{article}
\usepackage[T1]{fontenc}
\usepackage{amssymb, amsmath, amsthm}
\usepackage[all]{xy}
\usepackage{hyperref}
\usepackage{mathptmx,stmaryrd}

\DeclareFontEncoding{OT2}{}{}

\newcommand{\mr}[1]{\mathrm{#1}}
\newcommand{\mf}[1]{\mathfrak{#1}}
\newcommand{\mc}[1]{\mathcal{#1}}
\newcommand{\mb}[1]{\mathbb{#1}}

\newcommand{\Z}{\mb{Z}}
\newcommand{\Q}{\mb{Q}}
\newcommand{\R}{\mb{R}}
\newcommand{\zp}{\mb{Z}_p}
\newcommand{\qp}{\mb{Q}_p}

\newcommand{\C}{\mathcal{C}}
\newcommand{\D}{\mathcal{D}}
\newcommand{\T}{\mathcal{T}}

\newcommand{\cts}{\mr{cts}}

\newcommand{\La}{\Lambda}
\newcommand{\Lo}{\La^{\circ}}

\newcommand{\invlim}[1]{\varprojlim_{#1}\,}
\newcommand{\dirlim}[1]{\varinjlim_{#1}\,}
\newcommand{\ps}[1]{\llbracket #1 \rrbracket}

\newcommand{\cotimes}[1]{\,\hat{\otimes}_{#1} \,}

\DeclareMathOperator{\Hom}{Hom} 
 \DeclareMathOperator{\Gal}{Gal}
 
\DeclareMathOperator{\coker}{coker} \DeclareMathOperator{\Cone}{Cone}
\DeclareMathOperator{\Ch}{Ch} \DeclareMathOperator{\Maps}{Maps}
 \DeclareMathOperator{\Mod}{Mod}

\newtheorem{theorem}{Theorem}[subsection]
\newtheorem{proposition}[theorem]{Proposition}
\newtheorem{lemma}[theorem]{Lemma}
\newtheorem{corollary}[theorem]{Corollary}
\newtheorem*{thm}{Theorem}

\theoremstyle{definition}

\theoremstyle{remark}
\newtheorem{remark}[theorem]{Remark}
\newtheorem*{ack}{Acknowledgments}

\newcounter{countii}
\newenvironment{enum2}{\begin{list}{\alph{countii}.}{\usecounter{countii}
\setlength{\topsep}{3pt}
\setlength{\labelwidth}{0bp}
\setlength{\leftmargin}{18pt}
\setlength{\itemindent}{5pt}
\setlength{\itemsep}{0pt}
\setlength{\parsep}{0pt}}
}{\end{list}}

\numberwithin{equation}{section}

\begin{document}

\title{Reciprocity maps with restricted ramification}
\author{Romyar T. Sharifi}
\date{}
\maketitle

\begin{abstract}
	We compare two maps that arise in study of the cohomology of global fields 
	with ramification restricted to a finite set $S$ of primes.  One of these maps, which we call 
	an $S$-reciprocity map, interpolates the values of cup products in $S$-ramified cohomology.  In the
	case of $p$-ramified cohomology of the $p$th cyclotomic field for an odd prime $p$, 
	we use this to exhibit an intriguing relationship between particular 
	values of the cup product on cyclotomic $p$-units. 
	We then consider higher analogues of the $S$-reciprocity map and relate their cokernels to the graded quotients in
	augmentation filtrations of Iwasawa modules.
\end{abstract}

\section{Introduction}

The primary goal among several in this work is the comparison of two maps that arise in the Galois cohomology of global fields with restricted ramification, which might at first appear rather unrelated.  In this introduction, we introduce the two maps for a fixed global field.  In the body of the paper, we shall study the maps that they induce up towers of such fields.

Fix a prime $p$, a global field $F$ of characteristic not equal to $p$, and a finite set of primes $S$ of $F$ including all those above $p$ and any real infinite places.   We suppose that either $p$ is odd or $F$ has no real places.  We will use the terms \emph{$S$-ramified} and \emph{$S$-split} to mean unramified outside of the primes in $S$ and completely split at the primes in $S$, respectively.  Let $G_{F,S}$ denote the Galois group of the maximal $S$-ramified extension $F_S$ of $F$.

We will be concerned with three objects related to the arithmetic of $F$:
\begin{enum2}
	\item[$\bullet$] the pro-$p$ completion $\mc{U}_F$ of the $S$-units of $F$,
	\item[$\bullet$] the Galois group $\mf{X}_F$ of the 
	maximal $S$-ramified abelian pro-$p$ extension of $F$, and
	\item[$\bullet$] the Galois group $Y_F$ of the maximal unramified, $S$-split 
	abelian $p$-extension $H_F$ of $F$.
\end{enum2}
Each of these three groups has a fairly simple cohomological description.
That is, $\mc{U}_F$ is canonically isomorphic to the continuous cohomology group 
$H^1(G_{F,S},\zp(1))$ via Kummer theory, and $\mf{X}_F$ is canonically isomorphic to
its Pontryagin double dual expressed as $H^1(G_{F,S},\qp/\zp)^{\vee}$.  As a quotient of $\mf{X}_F$,
the group $Y_F$ is identified with 
the image of the Poitou-Tate map
$$
	H^1(G_{F,S},\qp/\zp)^{\vee} \to H^2(G_{F,S},\zp(1)),
$$
where $A^{\vee}$ denotes the Pontryagin dual of a topological abelian group $A$.
Alternatively, we may use class field theory and Kummer theory to identify
$Y_F$ with a subgroup of $H^2(G_{F,S},\zp(1))$.  That these provide the same identification
is non-obvious: see Theorem \ref{mapisnat} for a proof.

The group $G_{F,S}$ acts trivially on its abelian quotient $\mf{X}_F$ via conjugation.
The \emph{$S$-reciprocity map} for $F$ is then a homomorphism
$$
	\Psi_F \colon \mc{U}_F \to H^2(G_{F,S},\mf{X}_F(1))
$$
that interpolates values of the cup products
$$
	 H^1(G_{F,S},\Z/p^n\Z)  \times H^1(G_{F,S},\zp(1)) \xrightarrow{\cup} H^2(G_{F,S},\mu_{p^n})
$$
in the sense that if $\rho \in \Hom(\mf{X}_F,\Z/p^n\Z) \subseteq \mf{X}_F^{\vee}$ and $\rho^*$
denotes the induced map 
$$ 
	\rho^* \colon H^2(G_{F,S},\mf{X}_F(1)) \to H^2(G_{F,S},\mu_{p^n}),
$$
then for any $a \in \mc{U}_F$, we have
$$
	\rho \cup a = \rho^*(\Psi_F(a)).
$$ 
In fact, $\Psi_F$
may itself be viewed as left cup product with the element of $H^1(G_{F,S},\mf{X}_F)$
that is the quotient map $G_{F,S} \to \mf{X}_F$.
In the case that $F$ is abelian, the $S$-reciprocity map turns out to be of considerable arithmetic interest, its values relating to $p$-adic $L$-values of cuspidal eigenforms that satisfy congruences with Eisenstein series (see \cite{me-Lfn} and \cite{fk-proof}).  

We also have a second homomorphism
$$
	\Theta_F \colon \mc{U}_F \to H^1(G_{F,S},Y_F^{\vee})^{\vee}
$$
that can be defined as follows.  
Let $\mc{Y}$ denote the quotient of the group ring $\zp[Y_F]$ by the square
of its augmentation ideal, so that we have an exact sequence
$$
	0 \to Y_F \to \mc{Y} \to \zp \to 0
$$
of $\zp[G_{F,S}]$-modules, the $G_{F,S}$-cocycle determining the class of the extension 
being the quotient map $G_{F,S} \to Y_F$.
For each $v \in S$, choose a prime over $v$ in $F_S$ and thereby a homomorphism  $j_v \colon G_{F_v} \to G_{F,S}$ from the absolute Galois group of $F_v$.   The exact sequence is split as a sequence of 
$\zp[G_{F_v}]$-modules for the action induced by $j_v$. 

We have a cup product pairing
$$
	H^i(G_{F,S},\zp(1)) \times H^{2-i}(G_{F,S},Y_F^{\vee}) \to H^2(G_{F,S},Y_F^{\vee}(1)).
$$
As $G_{F,S}$ has $p$-cohomological dimension $2$, its target may be identified with the quotient of $H^2(G_{F,S},\mc{Y}^{\vee}(1))$ 
by the image of $H^2(G_{F,S},\mu_{p^{\infty}})$
under the maps induced by the Tate twisted dual sequence
$$
	0 \to \mu_{p^{\infty}} \to \mc{Y}^{\vee}(1) \to Y_F^{\vee}(1) \to 0.
$$
The dual of the local splitting of $\mc{Y} \to \zp$
and the sum of invariant maps of local class field theory provide maps
$$
	H^2(G_{F,S},\mc{Y}^{\vee}(1)) \to \bigoplus_{v \in S} H^2(G_{F_v},\mu_{p^{\infty}})
	\to \qp/\zp,
$$
and their composite is trivial on $H^2(G_{F,S},\mu_{p^{\infty}})$.  Thus, we have a well-defined pairing
$$
	H^i(G_{F,S},\zp(1)) \times H^{2-i}(G_{F,S},Y_F^{\vee}) \to \qp/\zp,
$$
and we consider the resulting map from its first term to the Pontryagin dual of its second.
For $i = 1$, this map is precisely $\Theta_F$.
For $i = 2$, it is identified with a map 
$$
	q_F \colon H^2(G_{F,S},\zp(1)) \to Y_F.
$$

An alternate, but equivalent, definition of these maps is given in Section \ref{fundex} (see also Section \ref{conncup2}).  Note that $\Theta_F$ and $q_F$ need not be canonical, as they can depend on
the choices of primes of $F_S$ over each $v \in S$ (or rather, the primes
of $H_F$ below them).

So that we may compare $\Psi_F$ and $\Theta_F$, let us examine their codomains.
As $G_{F,S}$ acts trivially on $\mf{X}_F$ and has $p$-cohomological dimension $2$, we have
$$
	H^2(G_{F,S},\mf{X}_F(1)) \cong H^2(G_{F,S},\zp(1)) \otimes_{\zp} \mf{X}_F.
$$ 
Moreover, there are canonical isomorphisms
$$
	H^1(G_{F,S},Y_F^{\vee})^{\vee} 
	\cong \Hom(\mf{X}_F,Y_F^{\vee})^{\vee}
	\cong \Hom(Y_F \otimes_{\zp} \mf{X}_F, \qp/\zp)^{\vee} 
	\cong Y_F \otimes_{\zp} \mf{X}_F,
$$ 
through which we identify the leftmost and rightmost groups in the equation.

We may now state our key result, the Iwasawa-theoretic generalization of which
is found in Theorem \ref{comparemaps}.

\begin{thm} \label{mainthm}
	The map $q_F$ is a splitting of the injection of $Y_F$ in $H^2(G_{F,S},\zp(1))$
	given by Poitou-Tate duality, and the diagram
	$$
		 \xymatrix@C=30pt@R=30pt{
			\mc{U}_F \ar[r]^-{\Psi_F} 
			\ar[rd]_{-\Theta_F} & H^2(G_{F,S},\zp(1)) \otimes_{\zp} \mf{X}_F 
			\ar[d]^{q_F \otimes \mr{id}} \\
			&  Y_F  \otimes_{\zp} \mf{X}_F
		}
	$$
	commutes.
\end{thm}

The case that $F = \Q(\mu_p)$ for an odd prime $p$ and $S$ contains only the prime over $p$ provides an entertaining
application that is studied in Section \ref{special}.  Consider the cup product pairing
$$
	(\ \, , \ ) \colon \mc{U}_F \times \mc{U}_F \to Y_F \otimes_{\zp} \mu_p.
$$
Let $\Delta = \Gal(F/\Q)$, let 
$\omega \colon \Delta \to \zp^{\times}$ be the Teichm\"uller character, and for odd $i \in \Z$, 
let $\eta_i \in \mc{U}_F$ denote the projection of $(1-\zeta_p)^{p-1}$ to the 
$\omega^{1-i}$-eigenspace $\mc{U}_F^{(1-i)}$ of $\mc{U}_F$.
Suppose that Vandiver's conjecture holds at $p$ and that two distinct eigenspaces 
$Y_F^{(1-k)}$ and $Y_F^{(1-k')}$ of $Y_F$ are nonzero for even $k, k' \in \Z$.
Then a simple argument using Theorem \ref{mainthm} tells us that 
$(\eta_{p-k},\eta_{k+k'-1})$ is nonzero if and only
if $(\eta_{p-k'},\eta_{k+k'-1})$ is nonzero (see Theorem \ref{equivvals}).  We find this rather
intriguing, as the two values lie in distinct eigenspaces of $Y_F \otimes_{\zp} \mu_p$
that would at first glance appear to bear little relation to one another.

In a sequel to this paper \cite{me-selmer}, we will apply Theorem \ref{mainthm} to the study of Selmer groups of modular representations.  In particular, we will consider a main conjecture of sorts for a dual Selmer group of the family of residual representations attached to cuspidal eigenforms that satisfy mod $p$ congruences with ordinary Eisenstein series, proving it under certain assumptions.   This work allows characters with parity opposite to that imposed in earlier well-known work of Greenberg and Vatsal \cite{gv}, which complicates the structure of the Selmer groups considerably.

In Section \ref{higher}, we turn to the study of \emph{higher $S$-reciprocity maps} arising from the filtration of the completed group ring $\Omega = \zp\ps{\mf{X}_F}$ by the powers of its augmentation ideal $I$.  
For $n \ge 1$, the $n$th higher $S$-reciprocity map 
$$
	\Psi_F^{(n)} \colon H^1(G_{F,S},\Omega/I^n(1)) \to H^2(G_{F,S},\zp(1)) \otimes_{\zp} I^n/I^{n+1}
$$
is defined as the connecting homomorphism arising from the exact sequence
$$
	0 \to I^n/I^{n+1} \to \Omega/I^{n+1} \to \Omega/I^n \to 0.
$$
The map $\Psi_F^{(1)}$ is just $\Psi_F$: for this, recall that $\Omega/I \cong \zp$ and $I/I^2 \cong \mf{X}_F$.

The higher reciprocity maps $\Psi_F^{(n)}$ have analogues $\Psi_{E/F}^{(n)}$ for every finite $S$-ramified
$p$-extension $E/F$ in which one replaces the $\zp$-algebra $\Omega$ by $\zp[\Gal(E/F)]$ and the augmentation ideal 
$I$ by the corresponding augmentation ideal. By the results of \cite{llsww}, the cokernel of $\Psi_{E/F}^{(n)}$ is isomorphic to 
the $n$th graded quotient in the augmentation filtration of $H^2(G_{E,S},\zp(1))$ for $G_{E,S} = \Gal(F_S/E)$. 
This isomorphism extends to general pro-$p$ extensions $E/F$ by using the
$S$-ramified Iwasawa cohomology of $E$ (see Theorem \ref{cokerSrec}).

Suppose now that $E/F$ is a $\zp$-extension.  We consider $Y_F$ and the Galois group $Y_E$ of the maximal unramified, $S$-split
abelian pro-$p$ extension of $E$.  Restriction defines a map $R_{E/F} \colon Y_E/IY_E \to Y_F$ that can be fit into a six term exact sequence as in \eqref{n=0seq}.  This sequence has its origins in work of Iwasawa \cite{iwasawa} and is found more directly
in the appendix to \cite{hs}.  We aim to understand the higher graded quotients $I^nY_E/I^{n+1}Y_E$ in terms of the higher $S$-reciprocity maps. 

To this end, we define in Section \ref{control} a restriction of $\Psi_{E/F}^{(n)}$ to a map
$$
	\psi_{E/F}^{(n)} \colon \mc{U}_{E/F}^{(n)} \to Y_F \otimes_{\zp} I^n/I^{n+1},
$$
where $\mc{U}_{E/F}^{(n)}$ consists of those elements of $H^1(G_{F,S},\Omega/I^n(1))$ with locally trivial image under 
$\Psi_{E/F}^{(n)}$.  
We then construct a map
$$
	 R_{E/F}^{(n)} \colon I^nY_E/I^{n+1}Y_E \to \coker \psi_{E/F}^{(n)}
$$ 
that is an analogue of $R_{E/F}$ for the higher graded quotients in the augmentation filtration of $Y_E$.

Theorem \ref{exseqgraded} provides a $6$-term exact sequence into which $R_{E/F}^{(n)}$ fits.  This significantly improves the main result of \cite{me-massey} (see Corollary 6.4 therein), stated there in terms of Massey products, that $R_{E/F}^{(n)}$ is an isomorphism
in the special case of perfect control that $S$ consists of a single prime that does not split at all in $E$.

\begin{ack}	
	The author's research was supported in part
	the National Science Foundation under Grant No.\  DMS-1801963. 
	Part of the research for this article took place 
	during stays at the Max Planck Institute for
	Mathematics, the Institut des Hautes \'Etudes Scientifiques, and the Fields Institute. 
	The author thanks these institutes for their hospitality.  The author thanks Bisi Agboola,
	Takako Fukaya, Ralph Greenberg, Kazuya Kato, and Bill McCallum for 
	stimulating conversations at various points in the writing.  He particularly thanks Meng Fai Lim for numerous
	helpful comments on some of the drafts of the paper.
\end{ack}

\section{Cohomology} \label{Galcoh}

The following section provides preliminaries on the Iwasawa-theoretic generalizations of Tate and Poitou-Tate duality that we require.  For the case of a complete local noetherian ring with finite residue field, these dualities are detailed in \cite{nekovar}.  The case of a general profinite ring can be found in \cite{lim}.  We use the standard conventions for signs of differentials on complexes, as in \cite[Chapter 1]{nekovar}.

\subsection{Connecting maps and cup products} \label{conncup1}

Let $G$ be a profinite group, and let $\La$ be a profinite ring.  We suppose that $\Lambda$ is a topological $R$-algebra for a commutative profinite ring $R$ in its center.  For later use, we fix a set $\mc{I}$ of open ideals of $\La$ that forms a basis of neighborhoods of $0$ in $\La$.

Let $\mc{T}_{\La,G}$ denote the category of
topological $\La$-modules endowed with a continuous $\La$-linear action of $G$, which is to say topological $\La\ps{G}$-modules, with morphisms the continuous $\La[G]$-module homomorphisms.
Let $\mc{C}_{\La,G}$ (resp., $\mc{D}_{\La,G}$)
denote the full subcategories of compact (resp., discrete) $\La$-modules.  Let $\mc{C}_{\La}$ (resp., $\mc{D}_{\La}$) denote the full subcategories of modules with trivial $G$-actions.  

Let
\begin{equation} \label{extension}
	0 \to A \xrightarrow{\iota} B \xrightarrow{\pi} C \to 0
\end{equation}
be a short exact sequence in $\mc{C}_{\La,G}$ or $\mc{D}_{\La,G}$.
Such a sequence gives rise to a short exact sequence
$$
	0 \to C(G,A) \to C(G,B) \to C(G,C) \to 0
$$
of inhomogeneous continuous $G$-cochain complexes (see Appendix \ref{cochcoh}), 
and the maps between the terms of the complexes are continuous with respect to the topologies on spaces of continuous maps of Lemma \ref{mapstop}.  In the case that the modules of \eqref{extension} are in $\mc{C}_{\La,G}$, we suppose in what follows that $G$ has
the property that the cohomology groups $H^i(G,M)$ of $C(G,M)$ 
are finite for every finite $\Z[G]$-module $M$.  The following is standard.

\begin{lemma}
	Let $s \colon C \to B$ be a continuous function that splits $\pi$ of \eqref{extension}.  Define 
	$$
		\partial^i \colon H^i(G,C) \to H^{i+1}(G,A)
	$$ 
	on the class of a cocycle $f$
	to be the class of $\iota^{-1} \circ d^i_B(s \circ f)$, where $d^i_B$ denotes the $i$th
	differential on $C(G,B)$.  The resulting sequence
	$$
		\cdots \to H^i(G,A) \to H^i(G,B) \to H^i(G,C) \xrightarrow{\partial^i} H^{i+1}(G,A) \to \cdots
	$$
	is exact in $\C_{\La}$ $($resp., $\mc{D}_{\La})$.
\end{lemma}

\begin{proof}
	The existence of $s$ in the case of $\C_{\La,G}$ is just \cite[Proposition 2.2.2]{rz}.  That
	$\partial^i$ is a homomorphism of $\La$-modules and that the sequence is exact are
	standard.    That the cohomology groups
	lie in the stated categories are Lemma \ref{cochdisc} and Proposition \ref{topctscoh}.
	That the homomorphisms are continuous in the case of $\C_{\La,G}$
	follows 
	quickly
	from Proposition \ref{topctscoh}.  
\end{proof}

Suppose now that $\pi$ has a continuous splitting $s \colon C \to B$ of $\La$-modules.
Then we obtain a $1$-cocycle
$$
	\chi \colon G \to \Hom_{\La,\cts}(C,A)
$$
given by
\begin{equation} \label{cocycle}
	\chi(\sigma)(c) = \sigma s(\sigma^{-1}c) - s(c)
\end{equation}
for $c \in C$, the class of which is independent of $s$.  
Conversely, the Galois action on $B$ is prescribed by the cocycle using \eqref{cocycle}.  As usual, $s$ gives rise to a $\La$-module splitting 
$t \colon B \to A$ of $\iota$ such that $\iota \circ t + s \circ \pi$ is the identity.

The following is well-known (see, e.g., \cite[Proposition 3]{flood}).

\begin{lemma}
	Endow $\Hom_{\La,\cts}(C,A)$ with the compact-open topology.
	\begin{enumerate}
		\item[a.] The group $\Hom_{\La,\cts}(C,A)$ is an object of $\mc{T}_{R,G}$.
		\item[b.] The canonical evaluation map $C \times \Hom_{\La,\cts}(C,A) \to A$
		is continuous.
		\item[c.] The $1$-cocycle $\chi \colon G \to \Hom_{\La,\cts}(C,A)$ is continuous.
	\end{enumerate}
\end{lemma}

We wish to compare our connecting homomorphisms with cup products. 
By Lemma \ref{cupprod}, we have continuous cup products
$$
	C^i(G,\Hom_{\La,\cts}(C,A)) \times C^j(G,C)
	\xrightarrow{\cup} C^{i+j}(G,A)
$$
that satisfy $\lambda(\rho \cup f) = \rho \cup \lambda f$ for all $\lambda \in \La$, 
$\rho \in C^i(G,\Hom_{\La,\cts}(C,A))$, and $f \in C^j(G,C)$.
For $\rho = \chi$ (with $i = 1$), the cup product $\chi \cup f \in C^{j+1}(G,A)$
may be described explicitly by the $(j+1)$-cochain
$$
	(\chi \cup f)(\sigma,\tau) = \chi(\sigma) \sigma f(\tau)
$$
for $\sigma \in G$ and $\tau$ a $j$-tuple in $G$.  We set $\tilde{\chi}(f) = \chi \cup f$.  

\begin{remark}
	For a cochain complex $X$ and $j \in \Z$, the complex $X[j]$ is that with $X[j]^i = X^{i+j}$
	and differential $d_{X[j]}$ on $X[j]$ that is $(-1)^j$ times the differential $d_X$ on $X$.  
\end{remark}

Since $d(\chi \cup f) = (-1)^j\chi \cup df$, we have maps of sequences of compact or discrete $\La$-modules 
$$
	 \tilde{\chi} \colon C(G,C) \to C(G,A)[1]
$$
The map $\tilde{\chi}$ commutes with the differentials of these complexes up to the sign $(-1)^{j+1}$ in degree $j$.  The map $\tilde{\chi}$ may be made into a map of of complexes by adjusting the signs, but we have no cause to do this here.

\begin{lemma} \label{cupprodconn}
	For all $i \ge 0$, the cup product map $\tilde{\chi}^i \colon H^i(G,C) \to H^{i+1}(G,A)$ agrees with the 
	connecting homomorphism $\partial^i$.
\end{lemma}

\begin{proof}
	Let $f$ be an $i$-cocycle on $G$ with values in $C$, and set $g = s \circ f$.
	The coboundary of $g$ has values in $A$, and one sees immediately that
	$$
		dg(\sigma,\tau) 
		= \sigma s(f(\tau)) - s(\sigma f(\tau)) = \chi(\sigma)(\sigma f(\tau)) = 
		(\chi \cup f)(\sigma,\tau)
	$$
	for $\sigma \in G$ and $\tau$ an $i$-tuple of elements of $G$. 
\end{proof}

If $X \in \mc{T}_{\La,G}$ is locally compact, then its Pontryagin dual $X^{\vee} = \Hom_{\cts}(X,\R/\Z)$ with $G$-action given by $(g\phi)(x) = \phi(g^{-1}x)$ for $g \in G$, $x \in X$, and $\phi \in X^{\vee}$ is a locally compact object of $\mc{T}_{\Lo,G}$, where $\Lo$ denotes the opposite ring to $\La$.  Pontryagin duality induces an exact equivalence of categories
between $\mc{C}_{\La,G}$ and $\mc{D}_{\Lo,G}$.

Now suppose that the exact sequence \eqref{extension} is of modules in $\C_{\La,G}$,
and suppose $A$ and $C$ are endowed with the $\mc{I}$-adic topology.
The dual of \eqref{extension} fits in an exact sequence
\begin{equation} \label{dualextension}
	0 \to C^{\vee} \to B^{\vee} \to A^{\vee} \to 0
\end{equation}
in $\D_{\Lo,G}$.
The canonical isomorphism $\Hom_{\La}(C,A) \cong \Hom_{\Lo}(A^{\vee},C^{\vee})$ 
(see Proposition \ref{dualhom}) then produces a continuous $1$-cocycle
$$
	\chi^* \colon G \to \Hom_{\Lo}(A^{\vee},C^{\vee}),
$$
from $\chi$.
A direct computation shows that $\chi^*$ satisfies
$$
	\chi^*(\sigma)(\phi) = - \sigma t^*(\sigma^{-1} \phi) + t^*(\phi)
$$
for $\phi \in A^{\vee}$ and $t^* \colon A^{\vee} \to B^{\vee}$ 
the map associated to $t$.
That is, the cocycle $-\chi^*$ defines the class of the dual extension 
\eqref{dualextension} in 
$H^1(G,\Hom_{\Lo}(A^{\vee},C^{\vee}))$.  By Lemma \ref{cupprodconn},
we therefore have the following.

\begin{lemma}
	The cup product map $\widetilde{\chi^*}{}^i$ is the negative of the
	connecting homomorphism $H^i(G,A^{\vee}) \to H^{i+1}(G,C^{\vee})$.
\end{lemma}

\subsection{Global fields} \label{global}

Let $p$ be a prime number, and now suppose that $\La$ is a pro-$p$ ring.
Let $F$ be a global field of characteristic not equal to $p$, and let $S$ denote a finite set of primes of $F$ including those above $p$ and any real places of $F$.   Let $S^f$ (resp., $S^{\infty}$) denote the
set of finite (resp., real) places in $S$.
We use $G_{F,S}$ to denote the Galois group of the maximal unramified outside $S$ extension of $F$.  
Fix, once and for all, a local embedding of the algebraic closure of $F$ at a prime above each $v \in S$ and therefore a homomorphism $j_v \colon G_{F_v} \to G_{F,S}$, where $G_{F_v}$ denotes the absolute
Galois group of the completion $F_v$ of $F$ at the prime over $v$.  

For $M \in \mc{T}_{\La,G_{F,S}}$, we consider the direct sum of local cochain complexes
\begin{equation} \label{local_coch}
	C_{l}(G_{F,S},M) = \bigoplus_{v \in S^f} C_{}(G_{F_v},M) \oplus \bigoplus_{v \in S^{\infty}} 
	\widehat{C}(G_{F_v},M),
\end{equation}
where $\widehat{C}(G_{F_v},M)$ denotes the total complex of the Tate complex of continuous
cochains for $v \in S^{\infty}$. 

We will be interested in a shift of a usual cone complex.  To that end, recall that if $X$ and $Y$ are cochain complexes and 
$C = \Cone(f \colon X \to Y)[-1]$, then $C = X \oplus Y[-1]$ with differential
$$
		d_C(x,y) = (d_X(x),-f(x)-d_Y(y)).
	$$
	
We define the (modified) continuous compactly supported cochain complex of $M \in \mc{T}_{\La,G_{F,S}}$ by
\begin{equation} \label{c-supp_coch}
	C_c(G_{F,S},M) = \Cone(C(G_{F,S},M) \xrightarrow{\ell_S} C_l(G_{F,S},M))[-1],
\end{equation}
where $\ell_S = (\ell_v)_{v \in S}$ is the sum of the localization maps $\ell_v$ defined by 
$\ell_v(f) = f \circ j_v$.  
The $i$th cohomology group of $C_c(G_{F,S},M)$ (resp., $C_l(G_{F,S},M)$) 
is denoted $H^i_c(G_{F,S},M)$ (resp., $H^i_l(G_{F,S},M)$).   Note that if $M$ is finite, then
$H^i_l(G_{F,S},M)$ and $H^i(G_{F,S},M)$ are finite for all $i \in \Z$, 
and as a result, so are the $H^i_c(G_{F,S},M)$. 
 
For $T \in \C_{\La,G_{F,S}}$, we may then view each $H^i_*(G_{F,S},T)$ as an object of $\C_{\La}$,
where $*$ denotes either no symbol, $l$, or $c$.
We then have a long exact sequence
$$
	\cdots \to H^i_c(G_{F,S},T) \to H^i(G_{F,S},T) \to H^i_l(G_{F,S},T) \to H^{i+1}_c(G_{F,S},T) \to \cdots
$$
in $\C_{\La}$.  
Similarly, we may view each $H^i_*(G_{F,S},A)$ for $A \in \D_{\La,G_{F,S}}$ as an object of $\D_{\La}$,
and we obtain a corresponding long exact sequence for such an $A$.

Suppose that 
$$
	\phi \colon M \times N \to L
$$
for $M \in \mc{T}_{\La,G_{F,S}}$, $N \in \mc{T}_{\Lo,G_{F,S}}$, and $L \in \mc{T}_{R,G_{F,S}}$
is a continuous, $\La$-balanced, 
$G_{F,S}$-equivariant homomorphism. (Recall that $\phi$ being $\La$-balanced means that, 
viewing $M$ as a right $\La$-module, we have $\phi(m\lambda,n) = \phi(m,\lambda n)$ for all 
$m \in M$, $n \in N$, and $\lambda \in \La$.)
Similarly to Lemma \ref{cupprod}, we have compatible continuous, $\La$-balanced cup products 
\begin{eqnarray*}
	&C^i_l(G_{F,S},M) \times C^j_l(G_{F,S},N) \xrightarrow{\cup_l} C^{i+j}_l(G_{F,S}, L)&\\
	&C^i(G_{F,S},M) \times C^j_c(G_{F,S},N) \xrightarrow{\cup_c} C^{i+j}_c(G_{F,S}, L)&\\
	&C^i_c(G_{F,S},M) \times C^j(G_{F,S},N) \xrightarrow{{}_c\cup} C^{i+j}_c(G_{F,S},L)&
\end{eqnarray*}
of $R$-modules for all $i, j \in \Z$ (see \cite[Section 5.7]{nekovar}).  As in Lemma \ref{cupprod},
if we introduce bimodule structures on $M$, $N$, and $L$ that commute with the $G_{F,S}$-action, 
and if the pairing $\phi$ is also linear for the new left and right actions on $M$ and $N$, respectively, 
then the cup products are likewise linear for them.

Class field theory yields that $H^i_c(G_{F,S},T)$ (and $H^i_c(G_{F,S},A)$)
are trivial for $i > 3$ and
$$
	H^3_c(G_{F,S},\qp/\zp(1)) \xrightarrow{\sim} \qp/\zp.
$$
Taking the $\La$-balanced pairing $T^{\vee} \times T(1) \to \qp/\zp(1)$,
cup products then give rise to compatible isomorphisms in $\C_{\La}$
of Tate and Poitou-Tate duality (see \cite[Theorem 4.2.6]{lim}):
\begin{eqnarray*}
	& \beta_{l,T}^i \colon H^i_l(G_{F,S},T(1)) \to 
	H^{2-i}_l(G_{F,S},T^{\vee})^{\vee}\\
	& \beta_{T}^i \colon H^i(G_{F,S},T(1)) \to 
	H^{3-i}_{c}(G_{F,S},T^{\vee})^{\vee}\\
	& \beta_{c,T}^i \colon 
	H^i_{c}(G_{F,S},T(1)) \to  H^{3-i}(G_{F,S},T^{\vee})^{\vee}.
\end{eqnarray*}

\subsection{Connecting maps and cup products} \label{conncup2}

Suppose that we are given an exact sequence of compact or discrete modules 
as in \eqref{extension} (with $G = G_{F,S}$) and a continuous splitting $s \colon C \to B$ of $\La$-modules.  
Let
$\chi \colon G_{F,S} \to \Hom_{\La,\cts}(C,A)$
be the induced continuous $1$-cocycle.

We wish to define a certain Selmer complex for $B$ under the assumption that the class of $\chi$ is locally trivial in the sense that $\ell_v(\chi)$ is a $G_{F_v}$-coboundary for each $v \in S$.   For this, we need the following lemma, the proof of which is straightforward and left to the reader.

\begin{lemma} \label{gesplit} 
	Fix a continuous homomorphism $\varphi_v \colon C \to A$ of $\La$-modules.  Give $A$, $B$, and $C$
	the $G_{F_v}$-actions induced by $j_v$ and their $G_{F,S}$-actions.  The following are equivalent:
	\begin{enumerate}
		\item[(i)] $\ell_v(\chi) = d\varphi_v$,
		\item[(ii)] $s - \iota \circ \varphi_v$ is $G_{F_v}$-equivariant,
		\item[(iii)] $t + \varphi_v \circ \pi$ is $G_{F_v}$-equivariant.
	\end{enumerate}
\end{lemma}

For each $v \in S$, assume that $\ell_v(\chi)$ is a coboundary, and choose $\varphi_v$ as in Lemma \ref{gesplit}.  Note that the following construction can be affected by these choices, even on cohomology.  If $\Hom_{\La,\cts}(C,A)$ has trivial
$G_{F,S}$-action, then it is always possible to choose $\varphi_v = 0$.
We then define 
$$
	C_{f}(G_{F,S},B) = 
	\Cone(C_{}(G_{F,S},B) \xrightarrow{t_S} 
	C_{l}(G_{F,S},A))[-1],
$$ 
where $t_S$ is defined as the composition of $\ell_S$ with the map of complexes determined by $t_v = t + \varphi_v \circ \pi$ in its $v$-coordinate.  In that $\iota$ is locally split by $t_v$, the map induced by $t_S$ on cohomology is independent of the choices of $\varphi_v$.  

It should be noted that the complex $C_{f}(G_{F,S},B)$ depends in general upon the choices of local embeddings.  In fact, we have the following lemma.

\begin{lemma} \label{selcho}
	For $\sigma \in G_{F,S}$, consider the homomorphisms $j'_v$ for $v \in S$ 
	that are defined on $\tau \in G_{F_v}$ by
	$j'_v(\tau) = \sigma j_v(\tau) \sigma^{-1}$, 
	and form the corresponding localization maps $\ell_v'$.  Define $\varphi_v' \colon C \to A$
	on $c \in C$ by
	$$
		\varphi_v'(c) = \sigma \varphi_v(\sigma^{-1}c) - \chi(\sigma)(c),
	$$
	and let $t_S'$ be the map induced in the $v$-coordinate by the composition
	$t'_v \circ \ell_v'$, where $t'_v = t + \varphi_v' \circ \pi$.
	Then the diagram
	$$
		 \xymatrix{
			C_{}(G_{F,S},B) \ar[r]^-{t_S} \ar[d]^{\sigma^*} &
			C_{l}(G_{F,S},A) \ar[d]^{\sigma} \\
			C_{}(G_{F,S},B) \ar[r]^-{t_S'} &
			C_{l}(G_{F,S},A) 
		}
	$$
	of maps of complexes commutes.  
	Here, $\sigma^*$ is the isomorphism induced by the standard action of $\sigma$ on cochains,
	and the right-hand vertical map is induced by the action of $\sigma$ on coefficients,
	where the action of $G_{F_v}$ on $A$ 
	is understood to be attained through $j_v$ on the upper
	row and $j_v'$ on the lower.
\end{lemma}

\begin{proof}
	Let $v \in S$ and $\tau \in G_{F_v}$.  
	We claim first that the above choice of $\varphi'_v$ makes $t'_v$ into a $G_{F_v}$-equivariant map for the 
	$G_{F_v}$-actions induced by $j'_v$.  By Lemma \ref{gesplit}, it suffices to check the Galois-equivariance
	of $s'_v = s-\iota \varphi'_v$, given the Galois-equivariance of $s_v = s-\iota\varphi_v$.  Set $\tau'_v = j'_v(\tau)$ and 
	$\tau_v = j_v(\tau)$.   As maps on $C$, we have
	$$
		s'_v\tau'_v = s\tau'_v - \iota(\sigma\varphi_v\sigma^{-1}\tau'_v - \chi(\sigma)\tau'_v)
		= \sigma s_v\sigma^{-1}\tau'_v
		= \sigma s_v\tau_v\sigma^{-1}
		= \tau'_v \sigma s_v\sigma^{-1},
	$$
	and we then check
	$$
		\sigma s_v \sigma^{-1} = \sigma s \sigma^{-1} - \iota\sigma\varphi_v\sigma^{-1}
		= s + \iota \chi - \iota\sigma \varphi_v \sigma^{-1} = s'_v.
	$$

	We verify commutativity of the diagram.
	For $f \in C^i(G_{F,S},B)$, and now taking $\tau$ to be an $i$-tuple in $G_{F_v}$, we have
	$$
		(t'_v \circ \ell'_v \circ \sigma^*(f))(\tau)
		= (t'_v \circ \sigma^*(f))(\sigma \tau_v\sigma^{-1})
		= (t \sigma + \sigma \varphi_v  \pi - \chi(\sigma)\sigma \pi )(f(\tau_v)),
	$$
	and $\iota$ applied to the latter value is the following map applied to $f(\tau_v)$:
	$$
		\iota t \sigma + \sigma \iota\varphi_v  \pi 
		+ s  \pi \sigma - \sigma s  \pi
		= \sigma \circ (1-s\pi +\iota\varphi_v\pi)
		= \sigma \circ (\iota t + \iota\varphi_v\pi)
		= \iota \circ \sigma \circ t_v.
	$$
	In other words, we have
	$$
		(t'_v \circ \ell'_v \circ \sigma^*(f))(\tau) = (\sigma \circ t_v) (f(\tau_v))
		= (\sigma \circ t_v \circ \ell_v(f))(\tau).
	$$
\end{proof}

In particular, we see that, even when the Galois action on $\Hom_{\La,\cts}(C,A)$ is trivial, the diagram 
in Lemma \ref{selcho} (and the resulting diagram on cohomology) would not have to commute if instead we took our canonical choices $\varphi_v = \varphi_v' = 0$, since we could have $\chi(\sigma) \neq 0$.  In other words, the Selmer complex depends upon our choice of local embeddings. 

For later purposes, it is notationally convenient  for us to work with the Tate twist of the sequence \eqref{extension}.  We have the following exact sequences, with the third requiring the assumption that
each $\ell_v(\chi)$ is a coboundary:
\begin{eqnarray}
	&0 \to C_{}(G_{F,S},A(1)) \to C_{}(G_{F,S},B(1)) \to C_{}(G_{F,S},C(1)) \to 0& \\
	&0 \to C_{c}(G_{F,S},A(1)) \to C_{c}(G_{F,S},B(1)) \to C_{c}(G_{F,S},C(1)) \to 0 &\\
	& 0 \to C_{c}(G_{F,S},A(1)) \to C_{f}(G_{F,S},B(1)) \to C_{}(G_{F,S},C(1)) \to 0.& 
	\label{Selmer_seq}
\end{eqnarray}
These yield connecting maps
\begin{eqnarray*}
	& \kappa_B^i \colon H^i(G_{F,S},C(1)) \to H^{i+1}(G_{F,S},A(1)) &\\
	& \kappa_{c,B}^i  \colon H^i_c(G_{F,S},C(1)) \to H^{i+1}_{c}(G_{F,S},A(1))  &\\
	& \kappa_{f,B}^i  \colon H^i(G_{F,S},C(1)) \to H^{i+1}_{c}(G_{F,S},A(1)) &
\end{eqnarray*}
of $\La$-modules.

We have continuous cup products for any $i, j \in \Z$, given by
$$
	C^i(G_{F,S},\Hom_{\La,\cts}(C,A)) \times C^j(G_{F,S},C(1))
	\xrightarrow{\cup} C^{i+j}(G_{F,S},A(1)).
$$
As before, setting $\tilde{\chi}(f) = \chi \cup f$,
we obtain a map of sequences of $\La$-modules
$$
	 \tilde{\chi} \colon C_{}(G_{F,S},C(1)) \to C_{}(G_{F,S},A(1))[1] 
$$
that is a map of complexes up to the signs of the differentials.
In a similar manner, we have maps of sequences of $\La$-modules
\begin{eqnarray*}
	&\tilde{\chi}_c \colon  C_{c}(G_{F,S},C(1)) \to C_{c}(G_{F,S},A(1))[1]& \\
	&\tilde{\chi}_f \colon C_{}(G_{F,S},C(1)) \to C_{c}(G_{F,S},A(1))[1]&
\end{eqnarray*}
with the same property,
given explicitly as follows.  First, for $f = (f_1,f_2) \in C_{c}^i(G_{F,S},C(1))$, where
$f_1 \in C_{}^i(G_{F,S},C(1))$ and $f_2 \in C_{l}^{i-1}(G_{F_v},C(1))$, we set
$$
	\tilde{\chi}_c(f) = \chi \cup_c f = (\chi \cup f_1, -\ell_S(\chi) \cup f_2),
$$ 
where the second cup product is the sum of the local cup products at $v \in S$.  In the case that the class of
$\chi$ is locally a coboundary (of $\varphi_v$ at $v \in S$), 
we set $\varphi_S = (\varphi_v)_{v \in S}$ and let
$$
	\chi' = (\chi,-\varphi_S) \in C_{c}^1(G_{F,S},\Hom_{\La,\cts}(C,A))
$$
be the canonical cocycle lifting $\chi$.  For $f \in C_{}^i(G_{F,S},C(1))$, we then define
\begin{equation} \label{chif}
	\tilde{\chi}_f(f) = \chi' {}_c\!\cup f = (\chi \cup f,-\varphi_S \cup \ell_S(f)) \in C_{c}^{i+1}(G_{F,S},A(1)).
\end{equation}

\begin{proposition} \label{equalcup}
	One has $\kappa_B^i = \tilde{\chi}^i$, $\kappa_{c,B}^i = \tilde{\chi}_c^i$, and $\kappa_{f,B}^i = \tilde{\chi}_f^i$ for all $i \in \Z$.
\end{proposition}

\begin{proof}
	We treat the cases one-by-one, showing that the maps agree on cocycles.  
	The first equality is Lemma \ref{cupprodconn}.
	For the next, let $f = (f_1,f_2) \in C_{c}^i(G_{F,S},C(1))$ be a cocycle, which is to say that
	$df_1 = 0$ and $df_2 = -\ell_S(f_1)$.  Set $g = (g_1,g_2) = (s \circ f_1, s \circ f_2)$.  
	For $i \ge 1$, one then has that
	$$
		dg_1(\sigma,\tau) = \chi(\sigma)(\sigma f_1(\tau))
	$$
	for $\sigma \in G_{F,S}$ and $\tau$ an $i$-tuple in $G_{F,S}$
	and
	$$
		dg_2(\sigma,\tau) = -\ell_S(g_1)(\sigma,\tau)+\ell_S(\chi)(\sigma)(\sigma f_2(\tau)),
	$$
	where $\sigma \in G_{F_v}$ and $\tau$ is an $(i-1)$-tuple in $G_{F_v}$.
	It follows that
	$$
		dg = (d(g_1),-\ell_S(g_1)-d(g_2)) = (\chi \cup f_1,-\ell_S(\chi) \cup f_2),
	$$
	for $i \ge 1$.  
	
	For $i = 0$, we remark that $Z^0_{}(G_{F,S},C(1)) \to Z^0_{}(G_{F_v},C(1))$
	is injective for any prime $v$, so the only cocycles in $C^0_{c}(G_{F,S},C(1))$
	are sums of cocycles in $\hat{C}^{-1}(G_{F_v},B(1))$ for real places $v$.
	For any $i \le 0$ and any real $v$, we may then use repeated right cup product with a $2$-cocycle 
	with class generating $\hat{H}^2(G_{F_v},\zp)$ to reduce the question to $i \ge 1$, as right
	cup product commutes not only with $\kappa_{c,B}$ by definition of the connecting homomorphism and the
	compatibility of left cup product and coboundary on the level of cocycles, but also
	with $\tilde{\chi}_c$ by associativity of the cup product on the level of complexes.
	
	Finally, assume that the class of $\chi$ is locally trivial, 
	and let $f$ be a cocycle in $C^i_{}(G_{F,S},C(1))$.  We lift it to $(s \circ f, 0)$ in
	$$
		C^i_{f}(G_{F,S},B(1)) \cong C^i_{}(G_{F,S},B(1)) \oplus 
		C^{i-1}_{l}(G_{F_v},A(1)).
	$$
	Using the first case, its coboundary is then 
	$$
		(\chi \cup f, -t_S(s \circ f)) = (\chi \cup f, -\varphi_S \circ \ell_S(f))
		= (\chi \cup f, -\varphi_S \cup \ell_S(f)).
	$$ 
\end{proof}

\subsection{Duality} \label{duality}

Suppose now that our exact sequence \eqref{extension} is in $\C_{\La,G_{F,S}}$.
We will assume that $A$ and $C$ are endowed with the $\mc{I}$-adic
topology and that  $\pi$ has a (continuous) splitting $s \colon C \to B$ of $\La$-modules.
The map of Proposition \ref{dualhom} provides a cocycle
$$
	\chi^* \colon G_{F,S} \to \Hom_{\Lo}(A^{\vee},C^{\vee})
$$
attached to $\chi$.
The cocycle $-\chi^*$ defines $B^{\vee}$ as an extension of $C^{\vee}$ by $A^{\vee}$.   
Let 
$$
	\nu_B^i \colon H^{i+1}(G_{F,S},A^{\vee})^{\vee} \to H^i(G_{F,S},C^{\vee})^{\vee}
$$
denote the Pontryagin dual map of $\kappa_{B(1)^{\vee}}^i$, and define 
$\nu_{c,B}^i = (\kappa_{c,B(1)^{\vee}}^i)^{\vee}$ and $\nu_{f,B}^i = (\kappa_{f,B(1)^{\vee}}^i)^{\vee}$ similarly.
By Poitou-Tate duality and a careful check of signs, we have the following rather standard result.

\begin{lemma} \label{psixidual}
	The following diagram commutes:
	$$
		 \xymatrix@C=60pt{ 
		H^i(G_{F,S},C(1)) \ar[r]^{\kappa_{B}^i} \ar[d]^{\beta_C^i} & H^{i+1}(G_{F,S},A(1)) 
		\ar[d]^{\beta_A^{i+1}}\\
		H^{3-i}_c(G_{F,S},C^{\vee})^{\vee} \ar[r]^{(-1)^{i+1}\nu^{2-i}_{c,B}} 
		& H^{2-i}_c(G_{F,S},A^{\vee})^{\vee},
		}
	$$
	Similarly, we have
	$$
		 \beta_{c,A}^{i+1} \circ \kappa_{c,B}^i = (-1)^{i+1}\nu_B^{2-i}
		 \circ \beta_{c,C}^i
	$$
\end{lemma}

\begin{proof}
	For $f \in H^i(G_{F,S},C(1))$ and $g \in H_c^{2-i}(G_{F,S},A^{\vee})$, we have
	$$
		((\beta_A^{i+1} \circ \kappa_{B}^i)(f))(g) = (\beta_A^{i+1}(\chi \cup f))(g)
		= g \,{}_c\!\cup (\chi \cup f) = (-1)^i (\chi^* \cup_c  g)\, {}_c\!\cup f,
	$$
	the latter equality by \cite[5.3.3.2-4]{nekovar}, while
	$$
		((\nu_{c,B}^{2-i} \circ \beta_C^i)(f))(g) 
		= \beta_C^i(f)(\kappa_{c,B(1)^{\vee}}^{2-i}(g))
		=  -(\chi^* \cup_c g)\, {}_c \!\cup f.
	$$
\end{proof}

We also require the following general lemma, the proof of which is an simple exercise.

\begin{lemma} \label{quasiiso}
	Suppose that we have an exact triangle of cochain complexes of $\La$-modules
	$$
		M \xrightarrow{(\iota,\iota')} N \oplus N' \xrightarrow{\pi+\pi'} O \to M[1] .
	$$
	Then the map of complexes
	$$
		\Cone{(M \xrightarrow{\iota} N)} \xrightarrow{(-\pi,\iota')} \Cone{(N' \xrightarrow{\pi'} O)}
	$$
	is a quasi-isomorphism.
\end{lemma}

Applying Lemma \ref{quasiiso} to the exact triangle
$$
	C(G_{F,S},B(1))[-1] \to C_l(G_{F,S},B(1))[-1] \to C_c(G_{F,S},B(1)) \to C(G_{F,S},B(1))
$$
(noting that $B = A \oplus C$ locally), we see that $C_{f}(G_{F,S},B(1))$ is quasi-isomorphic to
\begin{equation} \label{selmercone}
	\Cone(C_{l}(G_{F_v},C(1))[-1] \to C_{c}(G_{F,S},B(1))).
\end{equation}

Assuming that $\chi$ is locally trivial, we have that $\chi^*$ is locally trivial and
$$
	C_{f}(G_{F,S},B^{\vee}) = 
	\Cone(C_{}(G_{F,S},B^{\vee}) \xrightarrow{(s^{\vee})_S} 
	C_{l}(G_{F_v},C^{\vee}))[-1].
$$
The following is now a consequence of Tate and Poitou-Tate duality.

\begin{proposition} \label{thetadual}
	There are natural isomorphisms of topological $\La$-modules
	$$
		\beta_{f,B}^i \colon H^i_f(G_{F,S},B(1)) \to H^{3-i}_f(G_{F,S},B^{\vee})^{\vee}
	$$
	that fit into a commutative diagram
	$$
	\footnotesize
		 \xymatrix@C=7pt{
			\cdots \ar[r] &
			H^i_c(G_{F,S},A(1)) \ar[r] \ar[d]^{\beta_{c,A}^i} &
			H^i_f(G_{F,S},B(1)) \ar[r] \ar[d]^{\beta_{f,B}^i} &
			H^i(G_{F,S},C(1)) \ar[r]^{\kappa_{f,B}^i} \ar[d]^{\beta_C^i} & 
			H^{i+1}_c(G_{F,S},A) \ar[r] \ar[d]^{\beta_{c,A}^{i+1}} & \cdots\\
			\cdots \ar[r] &			
			H^{3-i}(G_{F,S},A^{\vee})^{\vee} \ar[r] &
			H^{3-i}_f(G_{F,S},B^{\vee})^{\vee} \ar[r] & 
			H^{3-i}_c(G_{F,S},C^{\vee})^{\vee} \ar[r] & 
			H^{2-i}(G_{F,S},A^{\vee})^{\vee} \ar[r] & \cdots,
		}
	$$
	where the lower connecting map is $(-1)^{i+1}\nu_{f,B}^{2-i}$.
\end{proposition}

\subsection{Iwasawa cohomology} \label{iwacoh}

Let $K$ be a Galois extension of $F$ that is $S$-ramified over a finite extension of $F$ in $K$, and set $\Gamma = \Gal(K/F)$.  
Let $\Omega$ denote the maximal
$S$-ramified extension of $K$.
Let $R$ be a commutative pro-$p$ ring, and let $\mf{A}$ be a topological basis of open ideals of $R$.
We set $\La = R\ps{\Gamma}$, which is itself a pro-$p$ ring.
Typically, we will simply use $S$ to denote the set of primes $S_E$ in any extension $E$ of $F$ lying above those in $S$, and $G_{E,S}$ will denote the Galois group of the maximal $S$-ramified extension of $E$.  
If $E$ is contained in $K$, we set $\mc{G}_{E,S} = \Gal(\Omega/E)$.
We let $\La^{\iota}$ denote $\La$ viewed as a $\La[\mc{G}_{F,S}]$-module with $\La$ acting by left multiplication and $\sigma \in \mc{G}_{F,S}$ acting by right multiplication by the image of $\sigma^{-1}$ in $\Gamma$.

Let us review a few results that may be found in \cite{lim}, extending work of Nekovar \cite{nekovar}.
Suppose that $T \in \C_{\La,\mc{G}_{F,S}}$.  We define the Iwasawa cochain complex as
$$
	C_{*,S}(K,T) = C_*(\mc{G}_{F,S},\La^{\iota} \cotimes{R} T)
$$
where $*$ again denotes no symbol, $c$, or $l$, and where $\cotimes{R}$ denotes the completed tensor product.  
Via a version of Shapiro's lemma for cochains,
we have
$$
	C_{*,S}(K,T) \cong \invlim{E \subset K} C_*(\mc{G}_{E,S},T),
$$
with the inverse limit taken over number fields $E$ containing $F$ with respect to corestriction maps.  We may replace $\mc{G}_{E,S}$ with $G_{E,S}$ in the above and obtain a quasi-isomorphic complex, but the above isomorphism makes the $\La$-module
structure more transparent, since $T$ is only assumed to be a $G_{E,S}$-module for sufficiently large $E$.
We may view the Iwasawa cohomology group $H^i_{*,S}(K,T)$ as an object of
$\mc{C}_{\La}$ by endowing it with the initial topology.
Since $T$ is compact, we have a canonical isomorphism
$$
	H^i_{*,S}(K,T) \xrightarrow{\sim} \invlim{E \subset K} H^i_*(\mc{G}_{E,S},T)
$$
of $\La$-modules.

Similarly, for $A \in \mc{D}_{\La,G_{F,S}}$, we can and will make the identification
$$
	C_*(G_{K,S},A) \cong C_*(\mc{G}_{F,S},\Hom_{R,\cts}(\La,A)),
$$
where we view $\La$ here as a right $\La$-module via right multiplication and a $\mc{G}_{F,S}$-module via
left multiplication and put the discrete topology on the $\La[\mc{G}_{F,S}]$-module $\Hom_{R,\cts}(\La,A)$.  
Direct limits with respect to restriction maps provide
an isomorphism
$$
	\lim_{\substack{\rightarrow\\E \subset K}} C_{*}(\mc{G}_{E,S},A)
	\to C_{*}(G_{K,S},A)
$$
of complexes of discrete $\La$-modules.  Again, we may replace $\mc{G}_{E,S}$ with $G_{E,S}$ in this
isomorphism.  Direct limits being exact, we have resulting $\La$-module isomorphisms
$$
	 \lim_{\substack{\rightarrow\\E \subset K}} H^i_{*}(\mc{G}_{E,S},A) 
	 \xrightarrow{\sim} H^i_{*}(G_{K,S},A).
$$

Taking $T \in \mc{C}_{\Lambda,G_{F,S}}$, the pairing $T^{\vee} \times T(1) \to \mu_{p^{\infty}}$ induced
by composition gives rise to a continuous, perfect pairing
$$
	\Hom_{R,\cts}(\La,T^{\vee}) \times (\La^{\iota} \cotimes{R} T(1)) \to \mu_{p^{\infty}},
	\qquad (\phi,\lambda \otimes t) \mapsto \phi(\lambda)(t).
$$
Cup product then induces the following isomorphisms in $\C_{\La}$, much as before:
\begin{eqnarray*}
	& \beta_{l,T}^i \colon H^i_{l,S}(K,T(1)) \to H_l^{2-i}(G_{K,S},T^{\vee})^{\vee}\\
	& \beta_{T}^i \colon H^i_S(K,T(1)) \to H^{3-i}_{c}(G_{K,S},T^{\vee})^{\vee}\\
	& \beta_{c,T}^i \colon H^i_{c,S}(K,T(1)) \to  H^{3-i}(G_{K,S},T^{\vee})^{\vee}
\end{eqnarray*}
(see \cite[Theorem 4.2.6]{lim}).
These agree with the inverse limits of the duality maps over number fields in $K$.

For any complex $T$ in $\C_{R,G_{F,S}}$ 
on which the $G_{F,S}$-actions factor through $\Gamma$, let $T^{\dagger}$ be the complex 
in $\C_{\La,G_{F,S}}$ that is $T$ as a complex of topological $R$-modules, has $\Gamma$-actions 
those induced by the $G_{F,S}$-action on $T$, and has trivial $G_{F,S}$-actions.  

\begin{lemma} \label{twistcoh}
	Let $T \in \C_{R,G_{F,S}}$ be such that the $G_{F,S}$-action factors
	through $\Gamma$.
	We have the following natural isomorphisms in $\mc{C}_{\La}$:
	\begin{enumerate}
		\item[a.] $H^3_{c,S}(K,T(1)) \cong T^{\dagger}$,
		\item[b.] $H^2_{c,S}(K,T(1)) \cong H^2_{c,S}(K,R(1)) \cotimes{R} T^{\dagger}$,
		\item[c.] $H^2_S(K,T(1)) \cong H^2_S(K,R(1)) \cotimes{R} T^{\dagger}$ 
		if $p$ is odd or $K$ has no real places.
	\end{enumerate}
\end{lemma}

\begin{proof}
	By the above remark, we may replace $T$ by $T^{\dagger}$ in the cohomology groups
	in question, thereby assuming that $T \in \C_{\La,G_{F,S}}$ has a trivial $G_{F,S}$-action.  
	Write $T = \invlim{} T_{\alpha}$,
	where the finite 
	$T_{\alpha}$ run over the $\La$-quotients of $T$ by open submodules.  We then have
	$$
		H^i_{c,S}(K,T(1)) \cong \invlim{\alpha} H^i_{c,S}(K,T_{\alpha}(1)).
	$$
	Note that there is an exact sequence
	$$
		0 \to \zp^r \xrightarrow{\theta} \zp^s \to T_{\alpha} \to 0
	$$
	of compact $\zp$-modules with $r, s \ge 0$.  Since
	$H^3_{c,S}(K,\zp(1)) \cong \zp$ and $H^4_{c,S}(K,\zp(1))$ is trivial, the long exact sequence 
	in cohomology yields that the composite map
	$$
		T_{\alpha} \xrightarrow{\sim} H^3_{c,S}(K,\zp(1)) \otimes_{\zp} T_{\alpha} 
		\to H^3_{c,S}(K,T_{\alpha}(1))
	$$	
	of $\La$-modules is an isomorphism.
	Part a then follows by taking inverse 
	limits of these compatible isomorphisms.
	
	Next, since the map $H^3_{c,S}(K,\zp^r(1)) \xrightarrow{\theta} H^3_{c,S}(K,\zp^s(1))$
	is injective, $H^2_{c,S}(K,T_{\alpha}(1))$ is isomorphic
	to the cokernel of the map
	$$
		H^2_{c,S}(K,\zp(1)) \otimes_{\zp} \zp^r \xrightarrow{\mr{id} \otimes \theta} 
		H^2_{c,S}(K,\zp(1)) \otimes_{\zp} \zp^s.
	$$
	Thus, the canonical map
	$$
		H^2_{c,S}(K,\zp(1)) \otimes_{\zp} T_{\alpha} \to H^2_{c,S}(K,T_{\alpha}(1))
	$$
	is an isomorphism, and therefore so is
	$$
		H^2_{c,S}(K,\zp(1)) \cotimes{\zp} T \to H^2_{c,S}(K,T(1)).
	$$
	This also holds for $T = R$, so the composite map
	$$
		H^2_{c,S}(K,\zp(1)) \cotimes{\zp} T \xrightarrow{\sim} 
		H^2_{c,S}(K,\zp(1)) \cotimes{\zp} R \cotimes{R} T \to H^2_{c,S}(K,R(1)) \cotimes{R} T
	$$
	is an isomorphism as well, from which part b follows.
	
	Finally, if $p$ is odd or $K$ has no real places, then $H^3_S(K,\zp(1)) = 0$, and part c follows by the 
	analogous argument.
\end{proof}

Suppose now that we have an exact sequence
$$
	0 \to A \xrightarrow{\iota} B \xrightarrow{\pi} C \to 0
$$
in $\C_{R,\mc{G}_{F,S}}$ and that 
$\pi$ has a (continuous) splitting $s \colon C \to B$ of $R$-modules.
Again, we have a continuous $1$-cocycle 
$$
	\chi \colon \mc{G}_{F,S} \to \Hom_{R,\cts}(C,A).
$$
On cohomology, we again have connecting maps
\begin{eqnarray*}
	& \kappa_{B}^i \colon H^i_S(K,C(1)) \to H^{i+1}_S(K,A(1)) &\\
	& \kappa_{c,B}^i  \colon H^i_{c,S}(K,C(1)) \to H^{i+1}_{c,S}(K,A(1))  &
\end{eqnarray*}
of compact $\La$-modules.
If we require that each $\ell_v(\chi)$ be a local coboundary upon restriction to $G_{K,S}$, 
then we also have
$$
	\kappa_{f,B}^i \colon H^i_S(K,C(1)) \to H^{i+1}_{c,S}(K,A(1)), 
$$
but in general it is just a map of compact $R$-modules.
Again, these agree with left cup product by $\chi$, which is to say, the
inverse limits of the maps $\tilde{\chi}$,
$\tilde{\chi}_c$ , and $\tilde{\chi}_f$ at the finite level, respectively.

Consider the dual exact sequence and the resulting cocycle
$$
	\chi^* \colon \mc{G}_{F,S} \to \Hom_R(A^{\vee},C^{\vee}).
$$
We again denote the Pontryagin duals of the resulting $i$th connecting maps by 
$\nu_{B}^i$, $\nu_{c,B}^i$ and, if $\chi$
is locally cohomologically trivial, $\nu_{f,B}^i$ on $G_{K,S}$-cohomology.

Suppose that $A$ and $C$ are endowed with the $\mf{A}$-adic topology.  As in Lemma \ref{psixidual}, 
we have the following.

\begin{lemma} \label{psixidual2}
	The maps $\kappa_{B}^i$ and $\kappa_{c,B}^i$  
	coincide with the maps $(-1)^{i+1}\nu_{c,B}^{2-i}$ and $(-1)^{i+1}\nu_{B}^{2-i}$, respectively,
	under Poitou-Tate duality.
\end{lemma}

We define
$$
	C_{f,S}(K,B) = C_f(\mc{G}_{F,S},\La^{\iota} \cotimes{R} B).
$$
Similarly, define
$$
	C_f(G_{K,S},B^{\vee}) = \dirlim{E \subset K} C_f(\mc{G}_{E,S},B^{\vee}).
$$
We have the following analogue of Proposition \ref{thetadual}.

\begin{proposition} \label{Selmer_duality}
	Suppose that $\chi$ is locally cohomologically trivial upon restriction to $G_{K,S}$.
	Then there are natural isomorphisms
	$$
		\beta_{f,B}^i \colon H^i_{f,S}(K,B(1)) \to 
		H^{3-i}_f(G_{K,S},B^{\vee})^{\vee}
	$$
	of topological $R$-modules that are compatible with $\beta_{c,A}^i$ and $\beta_C^i$ 
	in the sense of Proposition \ref{thetadual}, and we have
	$\beta_{c,A}^{i+1} \circ \kappa_{f,B}^{i} = (-1)^{i+1}\nu_{f,B}^{2-i} \circ \beta_C^i$.
\end{proposition}

\section{$S$-Reciprocity maps} \label{recmap}

In this section, we prove our main theorem from the introduction: see Theorem \ref{comparemaps}.  In Theorem \ref{equivvals}, we describe its application to values of cup product pairings.

We continue to let $K$ denote a Galois extension of $F$ that is $S$-ramified over a finite extension of $F$.
From now on, we assume that either $p$ is odd or that $K$ has no real places.  We set $R = \zp$ and $\La = \zp\ps{\Gamma}$, with
$\Gamma = \Gal(K/F)$ as before.  

\subsection{The fundamental exact sequences} \label{fundex}

We briefly run through the Iwasawa modules (i.e., $\Lambda$-modules) of interest.  First,  we use $\mf{X}_K$ denote the Galois group of the maximal $S$-ramified abelian pro-$p$ extension of $K$.  Let $Y_K$ denote the Galois group of the maximal unramified abelian pro-$p$ extension of $K$ in which all primes above those in $S$ split completely.   Let $\mc{U}_K$ be the inverse limit under norm maps of the $p$-completions of the $S$-units in number fields in $K$.  

The map $\beta_{c,\zp}^2$ induced by cup product is a canonical isomorphism
$$
	H^2_{c,S}(K,\zp(1)) \xrightarrow{\sim} H^1(G_{K,S},\qp/\zp)^{\vee},
$$
allowing us to identify the group on the left with $\mf{X}_K$.
Moreover, the kernel of the natural map
$$
	H^2_{c,S}(K,\zp(1)) \to H^2_S(K,\zp(1))
$$
is exactly the kernel of the restriction map $\mf{X}_K \to Y_K$.  Therefore, we obtain a canonical injection
$$
	\iota_K \colon Y_K \to H^2_S(K,\zp(1)).
$$
Finally, Kummer theory allows us to canonically identify $\mc{U}_K$ with $H^1_S(K,\zp(1))$.

Let $\mc{X}$ denote the quotient of $\zp\ps{\mf{X}_K}$ by the square of its augmentation ideal $I_{\mf{X}_K}$.  With the standard isomorphism $\mf{X}_K \xrightarrow{\sim} I_{\mf{X}_K}/I_{\mf{X}_K}^2$ that takes an element $\sigma$ to $(\sigma-1) \bmod I_{\mf{X}_K}^2$, this gives rise to the exact sequence
\begin{equation} \label{Sramseq}
	0 \to \mf{X}_K \to \mc{X} \to \zp \to 0
\end{equation}
of compact $\zp\ps{G_{K,S}}$-modules, where $\mf{X}_K$ acts by left multiplication on $\mc{X}$ and trivially on
$\mf{X}_K$ and $\zp$.  The cocycle $\mu \colon G_{K,S} \to \mf{X}_K$ that this sequence defines, which is a homomorphism since $\mf{X}_K$ has trivial $G_{K,S}$-action, is just the restriction map on Galois groups. 

By part c of Lemma \ref{twistcoh},
our connecting map $\kappa_{\mc{X}}^1$ is a homomorphism
$$
	\Psi_K \colon \mc{U}_K \to H^2_S(K,\zp(1)) \cotimes{\zp} \mf{X}_K
$$
that we call the $S$-reciprocity map for $K$.  Since $\mu \in H^1(G_{K,S},\mf{X}_K)^{\mc{G}_{F,S}}$, the equality $\kappa_{\mc{X}}^1 = \tilde{\mu}^1$ and Galois equivariance of cup products imply that the $S$-reciprocity map $\Psi_K$ is a homomorphism of $\La$-modules.

\begin{remark}
    With an additional hypothesis, we may give $\mc{X}$ a continuous, $\zp$-linear $\mc{G}_{F,S}$-action that turns \eqref{Sramseq} into an exact sequence of $\Lambda$-modules.
    Let $M_K$ be the maximal $S$-ramified abelian pro-$p$ extension of $K$.
    Suppose that there exists a continuous homomorphism $\alpha \colon \Gamma \to \Gal(M_K/F)$ splitting the restriction map (e.g., that $\Gamma$ is free pro-$p$).  As topological spaces, one then has
    $$
    	\Gal(M_K/F) \cong \mf{X}_K \times \Gamma
    $$
    via the inverse of the map that takes $(x,\gamma)$ to $x\alpha(\gamma)$.
    Using this identification, we define a map $p \colon \Gal(M_K/F) \to \mf{X}_K$ by 
    $p(x,\gamma) = x$.  It is a cocycle as $\alpha$ is a homomorphism.
    The cocycle on $\mc{G}_{F,S}$ resulting from inflation gives rise to a continuous $\mc{G}_{F,S}$-action on 
    $\mc{X}$, extending the natural $G_{K,S}$-action (and the conjugation action of $\mc{G}_{F,S}$ on $\mf{X}_K$).    
    The class of $p$ is independent of the splitting, which
    means that the class of the extension \eqref{Sramseq} is as well.
\end{remark}

The following lemma is easily proven by pushing out the exact sequence \eqref{Sramseq}.

\begin{lemma} \label{psicomp}
	Let $M \in \C_{\zp,\mc{G}_{F,S}}$ with trivial $G_{K,S}$-action, and write 
	$M \cong \invlim{\alpha} M_{\alpha}$ with $M_{\alpha}$ a finite
	quotient of $M$.  For
	$$
		h = (h_{\alpha})_{\alpha} \in \invlim{\alpha} H^1_{}(G_{K,S},M_{\alpha}),
	$$
	let 
	$$
		\tilde{h} \colon \mc{U}_K \cong H^1_S(K,\zp(1)) \to H^2_S(K,M(1)) \cong 
		H^2_S(K,\zp(1)) \cotimes{\zp} M
	$$
	be the composite map induced by the inverse limit under corestriction of left cup products with
	cohomology classes that restrict to $h_{\alpha}$.
	Then 
	$$
		(\mr{id} \otimes h) \circ \Psi_K = \tilde{h}.
	$$
\end{lemma}

As a quotient of $\mc{X}$, we obtain the extension
\begin{equation} \label{Yseq}
	0 \to Y_K \to \mc{Y} \to \zp \to 0.
\end{equation}
In this case, the cocycle $\chi \colon G_{K,S} \to Y_K$ defining the extension is unramified everywhere and $S$-split.
Applying parts a and b of Lemma \ref{twistcoh},
the first connecting homomorphism is then a homomorphism
$$
	\Theta_K = \kappa_{f,\mc{Y}}^1 \colon \mc{U}_K \to \mf{X}_K \cotimes{\zp} Y_K,
$$
and the second is
$$
	q_K = \kappa_{f,\mc{Y}}^2 \colon H^2_S(K,\zp(1)) \to Y_K.
$$

\begin{proposition} \label{theta2}
	The composition $q_K \circ \iota_K$ is the identity of $Y_K$. 
\end{proposition}

\begin{proof}
	By Proposition \ref{equalcup}, the map $q_K$ agrees with $\tilde{\chi}_f^2$.   Precomposing it
	with the map 
	$$
		H^2_{c,S}(K,\zp(1)) \to H^2_S(K,\zp(1)),
	$$
	the resulting map
	$$
		H^2_{c,S}(K,\zp(1)) \to H^3_{c,S}(K,Y_K(1))
	$$ 
	is exactly $\tilde{\chi}_c^2 = \kappa_{c,\mc{Y}}^2$, again by Proposition \ref{equalcup}.  By 
	Lemma \ref{psixidual2}, 
	it is Pontryagin dual (under Poitou-Tate duality) to the map 
	$$
		H^0(G_{K,S},Y_K^{\vee}) \to H^1(G_{K,S},\qp/\zp)
	$$ 
	given by $-\kappa_{\mc{Y}(1)^{\vee}}^0 = \widetilde{\chi^*}{}^0$, or in other words, the map 
	$Y_K^{\vee} \to \mf{X}_K^{\vee}$ that takes an element to its precomposition with 
	the projection $\mf{X}_K \to Y_K$.  That being said, $q_K \circ \iota_K$ is as stated.
\end{proof}

Let 
$$
	\mr{sw} \colon Y_K \cotimes{\zp} \mf{X}_K  \xrightarrow{\sim} \mf{X}_K \cotimes{\zp} Y_K 
$$
denote the standard isomorphism giving commutativity of the tensor product, i.e., that
swaps the two coordinates of simple tensors.

\begin{theorem} \label{comparemaps}
	The diagram
	$$
		 \xymatrix{
			\mc{U}_K \ar[r]^-{\Psi_K} \ar[rd]_{-\Theta_K} & H^2_S(K,\zp(1)) \cotimes{\zp} \mf{X}_K 
			\ar[d]^{\mr{sw} \circ (q_K \otimes \mr{id})} \\
			& \mf{X}_K \cotimes{\zp} Y_K 
		}
	$$
	commutes.
\end{theorem}

\begin{proof}
	The skew-symmetry of the cup product yields the commutativity of the outer square in the
	following diagram of maps:
	$$
		 \xymatrix{
		& H^2_S(K,\mf{X}_K(1)) \ar[rd]^{\tilde{\chi}_f^2} \ar[d] \\
		H^1_S(K,\zp(1)) \ar[ru]^{\tilde{\mu}^1} \ar[rd]_{-\tilde{\chi}_f^1} 
		& \mf{X}_K \cotimes{\zp} Y_K & H^3_{c,S}(K,\mf{X}_K \cotimes{\zp} Y_K(1)) \ar[l]_-{\sim} \\
		& H^2_{c,S}(K,Y_K(1)) \ar[ru]_{\tilde{\mu}_c^2} \ar[u]^{\wr},
		}
	$$ 
	Here, the outer maps are as in Section \ref{conncup2}, where $\mu$ is restriction to $\mf{X}_K$ as above.
	The downward arrow is the composition of the natural isomorphism 
	$$
		H^2_S(K,\mf{X}_K(1)) \cong H^2_S(K,\zp(1)) \cotimes{\zp} \mf{X}_K
	$$ 
	with $\mr{sw} \circ (q_K \otimes \mr{id})$.
	Recall that $q_K$ is $\tilde{\chi}_f^2$ on 
	$H^2_S(K,\zp(1))$ followed by the invariant map.  Since the leftward arrow is the
	invariant map, the upper-right triangle commutes.
	
	The upward arrow is the composition of the isomorphism
	$$
		H^2_{c,S}(K,Y_K(1)) \cong H^2_{c,S}(K,\zp(1)) \cotimes{\zp} Y_K
	$$
	with the natural isomorphism $H^2_{c,S}(K,\zp(1)) \cong \mf{X}_K$ of Poitou-Tate duality
	(see Lemma \ref{twistcoh}b).  In other words, noting that $(\mf{X}_K \cotimes{\zp} Y_K)^{\vee}$
	may be canonically identified with $\Hom_{\cts}(\mf{X}_K,Y_K^{\vee})$, the
	Pontryagin dual of the upward arrow is the natural identification
	$$
		\Hom_{\cts}(\mf{X}_K,Y_K^{\vee}) \to H^1(G_{K,S},Y_K^{\vee}).
	$$
	Let $\mc{Z} = \mc{X} \cotimes{\zp} Y_K$, and note that 
	$$
		0 \to \mf{X}_K \cotimes{\zp} Y_K \to \mc{Z} \to Y_K \to 0
	$$
	is still exact.
	By Lemma \ref{psixidual2}, the Pontryagin dual of the map $\tilde{\mu}_c^2 = \kappa_{c,\mc{Z}}^2$
	in the diagram is the map
	$$
		\tilde{\mu}^0 = -\kappa_{\mc{Z}^{\vee}(1)}^0 \colon H^0(G_{K,S},\Hom_{\cts}(\mf{X}_K,Y_K^{\vee}))
		\to H^1(G_{K,S},Y_K^{\vee}),
	$$
	that is cup product with the projection map $\mu \in H^1(G_{K,S},\mf{X}_K)$.
	The dual of the invariant map is just the identification of 
	$H^0(G_{K,S},\Hom_{\cts}(\mf{X}_K,Y_K^{\vee}))$ with $\Hom_{\cts}(\mf{X}_K,Y_K^{\vee})$.  
	The composite map is composition with the projection, so it is again the natural identification.  
	In other words, the lower-right triangle commutes, so the left half of the diagram commutes, 
	and that is what was claimed.
\end{proof}

In general, suppose that $L$ is an $S$-ramified abelian pro-$p$ extension of $K$, and set $G = \Gal(L/K)$.  The reciprocity map induces
$$
	\Psi_{L/K} \colon \mc{U}_K \to H^2_S(K,\zp(1)) \cotimes{\zp} G,
$$
equal to $(\mr{id} \otimes \pi_G) \circ \Psi_K$, where $\pi_G \colon \mf{X}_K \to G$ is the restriction map on Galois groups.  We again refer to $\Psi_{L/K}$ as an $S$-reciprocity map, for the extension $L/K$.  Note that $\Psi_{L/K} = \widetilde{\pi_G}$ in the notation of Lemma \ref{psicomp}.

Consider the localization map
$$
	\ell_S \colon H^2_S(K,\zp(1)) \cotimes{\zp} G \to \prod_{v \in S_K} G 
$$
obtained from the isomorphisms
$$
	H^2_{\mr{cts}}(G_{E_v},\zp(1)) \cong \zp
$$
for $v \in S_E$ and $E$ a number field containing $F$. Since $\prod_{v \in S_K} \zp$ is a free profinite $\zp$-module, the map
from $Y_K \cotimes{\zp} G$ onto the kernel of $\ell_S$ is an isomorphism.

By Proposition \ref{equalcup} and a standard description of the local reciprocity map in local class field theory,
the $v$-coordinate of the composition $\rho_{L/K} = \ell_S \circ \Psi_{L/K}$ is the composition of the natural maps
$$
	\mc{U}_K \to \invlim{E_v \subset K_v} \invlim{n} E_v^{\times}/E_v^{\times p^n}
$$ 
with the local reciprocity map for the (fixed) local extension defined by $L/K$ at $v \in S_K$ and the canonical injection of its Galois group into $G$.  We will let 
$\mc{U}_{L/K}$ denote the kernel of $\rho_{L/K}$.

Theorem \ref{comparemaps} now has following corollary.

\begin{corollary} \label{firstcor}
	For any $S$-ramified abelian pro-$p$ extension $L$ of $K$ with Galois group $G$, one has 
	a commutative diagram
	$$
		 \xymatrix@C=40pt@R=30pt{
			\mc{U}_{L/K} \ar[r]^-{\Psi_{L/K}} \ar[rd]_-{(\pi_G \otimes \mr{id}) \circ \Theta_K} & 
			Y_K \cotimes{\zp} G \ar[d]^{-\mr{sw}_G} \\
			& G \cotimes{\zp} Y_K, 
		}
	$$
	where $\mr{sw}_G$ is the standard isomorphism yielding commutativity of the tensor product.
\end{corollary}

This, in turn, has the following notable corollary for the extension defined by $Y_K$.

\begin{corollary} \label{antisym}
	Let $H_K$ be the maximal unramified, $S$-split abelian pro-$p$ extension of $K$.  The map 
	$$
		 \Psi_{H_K/K} \colon \mc{U}_K \to Y_K \cotimes{\zp} Y_K
	$$
	has antisymmetric image.
\end{corollary}

\begin{proof}
	We apply Corollary \ref{firstcor}.
	The map $\rho_{H_K/K}$ is a sum of local reciprocity maps $\rho_v \colon \mc{U}_K \to Y_K$ that
	are all trivial since the decomposition group at $v$ in $Y_K$ is trivial by definition.  
	Hence, we have $\mc{U}_{H_K/K} = \mc{U}_K$.  
	
	We claim that
	$$
		(\pi_{Y_K} \otimes \mr{id}) \circ \Theta_K = \Psi_{H_K/K}.
	$$ 
	To see this, first recall that $\Theta_K = \tilde{\chi}_f^1 = \kappa_{f,\mc{Y}}^1$ by Proposition \ref{equalcup}.
	As $\iota_K \circ \pi_{Y_K}$ is the canonical map $H^2_{c,S}(K,\zp(1)) \to H^2_S(K,\zp(1))$, 
	we can see by the formula in \eqref{chif} that $(\pi_{Y_K} \otimes \mr{id}) \circ \kappa_{f,\mc{Y}}^1 = 
	\kappa_{\mc{Y}}^1$. In turn, this equals $(1 \otimes \pi_{Y_K}) \circ
	\kappa_{\mc{X}}^1$, which is $\Psi_{H_K/K}$ by definition.
\end{proof}

\begin{remark}
	Let us check that $\Theta_F$ agrees with the likewise-denoted map of the introduction.  
	If we compose $\Theta_F$ for $i =1$ (or $q_F$ for $i = 2$) with the map 
	$H^i_c(G_{F,S},Y_F(1)) \to H^{3-i}(G_{F,S},Y_F^{\vee})^{\vee}$
	of Poitou-Tate duality, we obtain a map $f \mapsto (g \mapsto \chi^* \,{}_c\!\cup f \cup g)$.
	On the other hand, $\widetilde{\chi^*}{}_{\!\!\!f}^2$ induces a map
	$$
		H^2(G_{F,S},Y_F^{\vee}(1)) \to H^3_c(G_{F,S},\qp/\zp(1)) \xrightarrow{\sim} \qp/\zp
	$$ 
	which we may compose with the pairing
	\begin{equation} \label{pairrev}
		H^{2-i}(G_{F,S},\zp(1)) \times H^i(G_{F,S},Y_F^{\vee}) \to H^2(G_{F,S},Y_F^{\vee}(1))
	\end{equation}
	to obtain the same map $f \mapsto (g \mapsto \chi^* \,{}_c\!\cup f \cup g)$ in the dual. 
	 
	It thus suffices to see that the map $\rho \colon 
	H^2(G_{F,S},Y_F^{\vee}(1)) \to \qp/\zp$ described in the introduction
	agrees with $\widetilde{\chi^*}{}_{\!\!\!f}^2$, which by Proposition \ref{equalcup} and the discussion of
	Section \ref{duality} is the negative of the connecting homomorphism
	 $\kappa_{f,\mc{Y}(1)^{\vee}}^2$.  The map $\rho$ is given by lifting from
	$H^2(G_{F,S},Y_F^{\vee}(1))$ to $H^2(G_{F,S},\mc{Y}^{\vee}(1))$, applying restriction and local
	splittings (i.e., $(s^{\vee})_S$) to land in $H^2_l(G_{F,S},\qp/\zp(1))$, and then applying the 
	sum of invariant maps.  The map $H^2(G_{F,S},\mc{Y}^{\vee}(1)) \to H^2_l(G_{F,S},
	\qp/\zp(1))$ that is the connecting homomorphism arising from the Selmer complex is 
	$-(s^{\vee})_S$ by definition.  The map $\kappa_{f,\mc{Y}(1)^{\vee}}^2$ is given by lifting the 
	cocycle in $Z^2(G_{F,S},\mc{Y}^{\vee}(1))$ to $C^2_f(G_{F,S},\mc{Y}^{\vee}(1))$,
	taking its coboundary, and lifting the resulting element of
	$B^3_f(G_{F,S},\mc{Y}^{\vee}(1))$ to $Z^3_c(G_{F,S},\qp/\zp(1))$.  For this, we can
	lift to $Z^2_l(G_{F,S},\qp/\zp(1))$ and then map to
	$Z^3_c(G_{F,S},\qp/\zp(1))$.  That is,  $\kappa_{f,\mc{Y}(1)^{\vee}}^2 = -\widetilde{\chi^*}{}_{\!\!\!f}^2$ 
	results from
	\begin{multline*}
		H^2(G_{F,S},Y_F^{\vee}(1)) \twoheadleftarrow 
		H^2(G_{F,S},\mc{Y}^{\vee}(1)) \\ \xrightarrow{-(s^{\vee})_S}
		H^2_l(G_{F,S},\qp/\zp(1)) \to H^3_c(G_{F,S},\qp/\zp(1)),
	\end{multline*}
	which by its description is $-\rho$.  Thus, we have the desired equality.
\end{remark}

\subsection{A special case} \label{special}

Let $F = \Q(\mu_p)$, and let $\Delta = \Gal(\Q(\mu_p)/\Q)$.  Let $S = S_{\Q} = \{p, \infty\}$, and let $\zeta_p$
denote a primitive $p$th root of unity.  Let $\omega \colon \Delta \hookrightarrow \zp^{\times}$ denote the
Teichm\"uller character.  For a $\zp[\Delta]$-module $M$ and $j \in \Z$, we let $M^{(j)}$ denote its $\omega^j$-eigenspace.    Finally, set
$$
	\eta_i = \prod_{\delta \in \Delta} (1-\zeta_p^{\delta})^{\omega^{i-1}(\delta)} \in \mc{U}_F^{(1-i)}
$$
for  odd $i \in \Z$.  

We will consider the cup product pairing
$$
	(\ ,\ )_{p,F,S} \colon H^1(G_{F,S},\mu_p) \times H^1(G_{F,S},\mu_p) \to 
	Y_F \otimes \mu_p.
$$
of \cite{mcs}, which we can view as a pairing on the group of 
elements of $F^{\times}$ which have associated fractional ideals that are
$p$th powers times a power of the prime over $p$.

Let us say that an integer $k$ is irregular for $p$ if $k$ is even, $2 \le k \le p-3$, and
$p$ divides the $k$th Bernoulli number $B_k$.
For primes of index of irregularity at least $2$ that, for instance, satisfy Vandiver's conjecture, we have the following consequence of the antisymmetry expressed in Corollary \ref{antisym}.

\begin{theorem} \label{equivvals}
	Suppose that $k$ and $k'$ are irregular for $p$ with $k < k'$ and that $Y_F^{(k)}$ and $Y_F^{(k')}$ are trivial.  Then
	$(\eta_{p-k},\eta_{k+k'-1})_{p,F,S} = 0$
	if and only if
	$(\eta_{p-k'},\eta_{k+k'-1})_{p,F,S} = 0$.
\end{theorem}

\begin{proof}
	For any $r$ that is irregular for $p$, we let
	$\chi_{p-r} \colon Y_F \to \mu_p$ denote the Kummer character attached to $\eta_{p-r}$.
	Supposing that $Y_F^{(r)} = 0$, fix an 
	$\mb{F}_p$-basis of $\Hom(Y_F,\mu_p)$ which contains the nontrivial $\chi_{p-r}$ 
	and a dual basis of $Y_F/p$ containing elements $\sigma_r \in Y_F^{(1-r)}$ for each such character
	with $\chi_{p-r}(\sigma_r) = \zeta_p$ and $\chi_{p-r}$ trivial on all other elements of the basis.  
	
	Since
	$$
		\chi_{p-k}(\kappa_{\mc{Y}}^1(\eta_{k+k'-1})) = (\eta_{p-k},\eta_{k+k'-1})_{p,F,S} \in Y_F^{(1-k')} \otimes \mu_p,
	$$
	the coefficient of $\sigma_{k'} \otimes \sigma_k$ in the
	expansion of $\kappa_{\mc{Y}}^1(\eta_{k+k'-1})$ mod $p$ 
	in terms of the standard basis of the tensor product 
	is  $[\eta_{p-k},\eta_{k+k'-1}]_{k'}$,
	where 
	$$
		(\eta_{p-k},\eta_{k+k'-1})_{p,F,S} = \sigma_{k'} \otimes \zeta_p^{[\eta_{p-k},\eta_{k+k'-1}]_{k'}}.
	$$
	Similarly, the coefficient of $\sigma_k \otimes \sigma_{k'}$ is the analogously defined
	$[\eta_{p-k'},\eta_{k+k'-1}]_k$.  The antisymmetry of Corollary \ref{antisym} forces
	$$
		[\eta_{p-k},\eta_{k+k'-1}]_{k'} = -[\eta_{p-k'},\eta_{k+k'-1}]_k,
	$$
	and, in particular, the result.
\end{proof}

This phenomenon can be seen in the tables of the pairing values for $p < 25,\!000$ produced by the author and McCallum (see \cite{mcs}).  What is remarkable about Theorem
\ref{equivvals} is that it relates pairing values in distinct eigenspaces of $Y_F$.

\begin{remark}
Note that the condition for the vanishing of these pairing values appears in the statement of \cite[Theorem 5.2]{me-gal}.  We remark that there is a mistake in said statement that is rendered inconsequential by Theorem \ref{equivvals}.  That is, it was accidentally only assumed there that $(\eta_{p-k},\eta_{k+k'-1} )_{p,F,S} = 0$ for all $k' > k$, while this vanishing for all $k' \neq k$ is used in the proof (although at one point the condition that $k' > k$ is again written therein).
\end{remark}

\section{Higher $S$-reciprocity maps} \label{higher}

In this section, we consider higher $S$-reciprocity maps, which are generalized Bockstein maps related to those of \cite{llsww}. We employ them to study the graded quotients in the augmentation filtration of completely split Iwasawa modules over Kummer extensions. We provide an exact sequence into which these graded quotients fit in Theorem \ref{exseqgraded}.

As in earlier sections, let $p$ be a prime, let $F$ be a global field of non-$p$ characteristic, and let $S$ be a finite primes of $F$ containing the
primes over $p$ and all real places. We continue to suppose that $p$ is odd or $F$ has no real places.

Let $L$ be an $S$-ramified Galois extension of $F$.  We assume that $G = \Gal(L/F)$ is almost pro-$p$ in the sense that it contains an open normal pro-$p$ subgroup.  We now let $K$ be a Galois subextension of $L/F$ and set $H = \Gal(L/K)$ so that 
$$
	\Gamma = \Gal(K/F) \cong G/H.
$$

Let $R$ be a complete commutative local Noetherian $\zp$-algebra with finite residue field.  
Let $T$ be an $R\ps{G_{F,S}}$-module that is finitely generated over $R$. We suppose that either
\begin{enumerate}
	\item[i.] $T$ is $R$-projective, or
	\item[ii.] $L/F$ is a $p$-adic Lie extension.
\end{enumerate}

In addition to $\La = R\ps{\Gamma}$, we set $\Omega = R\ps{H}$.  
If $L/F$ is a $p$-adic Lie extension, then so are $L/K$ and $K/F$, and in particular $\La$ and $\Omega$ are noetherian.

\subsection{Duality for finitely generated modules}

In this subsection, we establish some necessary background, in particular by providing a description in Theorem \ref{cokerSrec} of graded quotients of second Iwasawa cohomology groups.

The following is a consequence of \cite[Proposition 1.6.5]{fk} and \cite[Propositions 4.1.1 and 4.1.3]{ls}. 

\begin{proposition}[Fukaya-Kato, Lim-Sharifi] \label{spectral-global}
	We have homological spectral sequences of finitely generated $\Lambda$-modules
	\begin{align*}
		&\mb{J} = \mb{J}(T) \colon && 
		J^2_{i,j} = H_i(H, H^{2-j}_S(L,T) ) 
		&&\Rightarrow&& J_{i+j} = H^{2-i-j}_S(K,T),\\
		&\mb{P} = \mb{P}(T) \colon &&
		P^2_{i,j} =  H_i(H,H^{2-j}_{l,S}(L,T))
		&&\Rightarrow&& P_{i+j} = H^{2-i-j}_{l,S}(K,T),\\
		&\mb{C} = \mb{C}(T) \colon &&
		C^2_{i,j} = H_i(H, H^{3-j}_{c,S}(L,T))
		&&\Rightarrow&& C_{i+j} = H^{3-i-j}_{c,S}(K,T),
	\end{align*}
	that are natural in $T$
	and an exact sequence of spectral sequences
	$$
		\cdots \to \mb{J} \to \mb{P} \to \mb{C} \to \mb{J}[1] \to \cdots,
	$$
	where $\mb{J}[1]$ is the spectral sequence that is the shift of $\mb{J}$ with its second page having $(i,j)$-term $J_{i,j-1}^2$.
\end{proposition}

Let $I$ denote the augmentation ideal of $\Omega$.  Fix $n \ge 0$, 
and suppose either that $T$ is $R$-flat or that both $\Omega/I^n$ and $I^n/I^{n+1}$ are $R$-flat.  
Consider the exact sequence of $\Omega[G_{F,S}]$-modules
$$
	0 \to I^n/I^{n+1} \otimes_R T \to \Omega/I^{n+1} \otimes_R T \to \Omega/I^n \otimes_R T \to 0.
$$
Let
$$
	\Psi_{L/K,T}^{(n)} \colon H^1_S(K,\Omega/I^n \otimes_R T) \to H^2_S(K,T) \cotimes{R} I^n/I^{n+1}
$$
be the resulting connecting map.  Here, we have used the isomorphism
$$
	H^2_S(K,I^n/I^{n+1} \otimes_R T) \cong H^2_S(K,T) \cotimes{R} I^n/I^{n+1}
$$
that exists as $G_{F,S}$ has $p$-cohomological dimension $2$ and $I^n/I^{n+1}$ has trivial $G_{K,S}$-action
(cf. Lemma \ref{twistcoh}).
Note that the map $\Psi_{L/K,T}^{(0)}$ is a map from the zero group to $H^2_S(K,T)$.

As a consequence of \cite[Remark 3.2.5]{llsww}, we have the following description of the cokernel
of $\Psi_{L/K,T}^{(n)}$.\footnote{For this, note that nothing in \cite{llsww} required that $H$ be topologically finitely generated,
aside from a desire not to write completed tensor products.}

\begin{theorem} \label{cokerSrec}
	For each $n \ge 0$, we have a natural isomorphism of $\Lambda$-modules
	$$
		\coker \Psi_{L/K,T}^{(n)} \cong \frac{I^n H^2_S(L,T)}{I^{n+1} H^2_S(L,T)}.
	$$
\end{theorem}

\begin{remark}
	We also have the analogous isomorphisms for our semilocal Iwasawa cohomology groups
	$H^2_{l,S}(L,T)$ and for the Iwasawa cohomology groups of extensions of a local field of residue characteristic $p$.
\end{remark}

We next review the two primary ingredients in the proof of Theorem \ref{cokerSrec} that we shall require.
For a compact $\Omega$-module $A$, we let
$$
	(I^n \cotimes{\Omega} A)^{\circ} = \ker(I^n \otimes_{\Omega} A \to I^nA)
$$  
where the map is induced by multiplication.
The statement on the cokernel in the following lemma is \cite[Lemma 3.2.3]{llsww} and is the first of two key ingredients
in the proof of Theorem \ref{cokerSrec}.   The statement on the kernel is new and is used in Section \ref{control}.

\begin{lemma} \label{bocklemma}
	Let $A$ be a compact $\Omega$-module, and consider the exact sequence
	\begin{equation} \label{bockstein}
		0 \to I^n/I^{n+1} \cotimes{R} A \to \Omega/I^{n+1} \cotimes{R} A \to \Omega/I^n \cotimes{R} A \to 0.
	\end{equation} 
	The connecting homomorphism
	$$
		\partial_n \colon H_1(H,\Omega/I^n \cotimes{R} A) \to A_H \cotimes{R} I^n/I^{n+1}
	$$
	in the $H$-homology of \eqref{bockstein}
	has cokernel isomorphic to $I^n A/I^{n+1} A$. 
	 If $I^n$ is $\Omega$-flat, then $\partial_n$ has kernel isomorphic to $(I^n \cotimes{\Omega} IA)^{\circ}$.
\end{lemma}

\begin{proof}
	Note that $(\Omega/I^m \cotimes{R} A)_H \cong A/I^mA$ for all $m \ge 0$.
	The cokernel of $\partial_n$ is then identified with the kernel of the
	quotient map $A/I^{n+1}A \to A/I^nA$, which is $I^nA/I^{n+1}A$.
	Since $A \cotimes{R} \Omega$ is $H$-acyclic for homology, we have Rthat	
	\begin{equation} \label{H1Omegan} 
		H_1(H,\Omega/I^n \cotimes{R} A) \cong 
		\ker ((I^n \cotimes{R} A)_H \to (\Omega \cotimes{R} A)_H)
		\cong (I^n \cotimes{\Omega} A)^{\circ}.
	\end{equation} 
	The map $\partial_n$ is induced by the identity on $I^n \otimes_{\Omega} A$.
	In particular, $\partial_n$ can be identified with a restriction of the map 
	$$
		I^n \cotimes{\Omega} A \to I^n \cotimes{\Omega} A_H.
	$$
	If $I^n$ is flat over $\Omega$, then the latter map has kernel equal to 
	$I^n \cotimes{\Omega} I A$, and on the subgroup $(I^n \cotimes{\Omega} A)^{\circ}$ it has the stated kernel.  
\end{proof}

The second key ingredient from \cite{llsww} is the commutative diagram of $\Lambda$-module homomorphisms
\begin{equation} \label{keysq}
	 \xymatrix{
		H^1_S(K,T \otimes_R \Omega/I^n) \ar[r]^-{\Psi_{L/K,T}^{(n)}} \ar@{->>}[d] & 
		H^2_S(K,T) \cotimes{R} I^n/I^{n+1} \ar[d]^-{\wr} \\
		H_1(H,H^2_S(L,T) \cotimes{R} \Omega/I^n) \ar[r]^-{\partial_n} & H^2_S(L,T)_H \cotimes{R} I^n/I^{n+1},
	}
\end{equation}
where $\partial_n$ is the connecting homomorphism \eqref{bockstein} for $A = H^2_S(L,T)$.  
We then apply Lemma \ref{bocklemma} to obtain Theorem \ref{cokerSrec}
as the isomorphism on cokernels of the two horizontal maps in \eqref{keysq}.

\subsection{Control theorem for higher graded quotients} \label{control}

We turn to the consideration of the main result of this section.  Suppose now that $R = \zp$ and $T = \zp(1)$.
We refer to 
$$
	\Psi_{L/K}^{(n)} = \Psi_{L/K,\zp(1)}^{(n)} \colon H^1_S(K,\Omega/I^n(1)) \to H^2_S(K,\zp(1)) \cotimes{\zp} I^n/I^{n+1}
$$ 
as the $n$th higher $S$-reciprocity map for $L/K$.  
In the case that $H = \mf{X}_K$, we could set 
$\Psi_K^{(n)} = \Psi_{L/K}^{(n)}$,
as in the case $n = 1$.  However, in this section, we assume that $H \cong \zp$.
Under this assumption, $I^n/I^{n+1} \cong H^{\otimes n}$ is free of rank $1$ over $\zp$, and the surjection $\Omega/I^{n+1} \to \Omega/I^n$ is $\zp$-split, since $\Omega/I^n$ is $\zp$-free.

Let us first review the case $n = 0$. We have the map $R_{L/K} \colon (Y_L)_H \to Y_K$ on completely split Iwasawa modules 
that is restriction on Galois groups,
the map $N_{L/K} \colon \mc{U}_L \to \mc{U}_K$ on inverse limits of $S$-unit groups induced by norm maps, 
and the map $\Sigma_{L/K}$
that is the inverse limit of sums of inclusion maps
$$
	\bigoplus_{u \in S_{K'}} \Gal(L'/K')_u \to \Gal(L'/K')
$$ 
of decomposition groups taken over number fields $L' \subset L$, where $K' = K \cap L'$.
By \cite[Corollary A.2]{hs} (in the case that the set of primes there is taken to be $S$), we have a canonical exact sequence
\begin{equation} \label{n=0seq}
	0 \to Y_L^H \otimes_{\zp} H \to \mc{U}_K/N_{L/K}\mc{U}_L \to \ker \Sigma_{L/K} \to (Y_L)_H \xrightarrow{R_{L/K}} Y_K
	\to \coker \Sigma_{L/K} \to 0
\end{equation}
of $\Lambda$-modules
that ties these groups and maps all together. In Theorem \ref{exseqgraded} below, we 
extend this to higher graded quotients of $Y_L$ in the augmentation filtration.
    
Now let us turn to the case of general $n$. For brevity, for any compact $\zp$-module $M$, let us set  
$$
	M\{n\} = M \cotimes{\zp} I^n/I^{n+1},
$$
which is of course noncanonically isomorphic to $M$ as a $\zp$-module.
For any $\Omega$-module $A$, let ${}_{I^n} A$ denote the submodule of elements of $A$ annihilated by $I^n$.
We have canonical isomorphisms
\begin{equation} \label{H1HZp}
	H_1(H,A \cotimes{R} \Omega/I^n) \cong (I^n \cotimes{\Omega} A)^{\circ} \cong ({}_{I^n} A)\{n\},
\end{equation}
the first by \eqref{H1Omegan}.
Moreover, since $H$ has $p$-cohomological dimension $1$, we have exact sequences
\begin{equation} \label{threeterm}
	0 \to E_{0,j}^2 \to E_j \to E_{1,j-1}^2 \to 0
\end{equation}
for $\mb{E} = (E_{i,j}^r,E_{i+j})$ any of the spectral sequences $\mb{C}$, $\mb{P}$, and $\mb{J}$.

\begin{lemma} \label{local_bound_inj}
	The semilocal connecting map
	$$
		\partial_l^{(n)} \colon {}_{I^n} H^2_{l,S}(L,\zp(1))\{n\} \to H^2_{l,S}(L,\zp(1))_H\{n\}
	$$
	as in Lemma \ref{bocklemma} with $A = H^2_{l,S}(L,\zp(1))$, and given the isomorphism of \eqref{H1HZp}, is injective.
\end{lemma}

\begin{proof}
	By Lemma \ref{bocklemma} and the assumption that $H \cong \zp$, we have
	$$
		\ker \partial_l^{(n)} \cong {}_{I^n}IH_{l,S}^2(L,\zp(1))\{n\}.
	$$
	From this, we see that the result for all $n$ is equivalent to the triviality of the $H$-invariant group of
	$I H_{l,S}^2(L,\zp(1))$.
	
	Note that
	$$
		H_{l,S}^2(L,\zp(1)) \cong \prod_{v \in S_K}
		\zp\ps{H/H_v}, 
	$$
	where $H_v$ denotes the decomposition group at a prime above $v$.
	If $H_v = 0$, then the submodule of $H$-fixed elements is trivial. 
	Otherwise $H/H_v$ is finite, and only the span of 
	the norm element in $\zp\ps{H/H_v}$ 
	is fixed by $H$, but the norm element is not contained in $I\zp\ps{H/H_v}$.
\end{proof}

Let $\mc{U}_{L/K}^{(n)}$ be the global kernel of the $n$th semilocal higher reciprocity map
$$
	H^1_S(K,\Omega/I^n(1)) \to H_{l,S}^2(K,\zp(1))\{n\}
$$
(i.e., the composition of $\Psi_{L/K}^{(n)}$ with the sum of local restriction maps), and let
$$
	\psi_{L/K}^{(n)} \colon \mc{U}_{L/K}^{(n)} \to Y_K\{n\}
$$
denote the map induced by the restriction of $\Psi_{L/K}^{(n)}$ to $\mc{U}_{L/K}^{(n)}$.

The exact sequence in the following is an extension of \cite[Theorem 6.3]{me-massey}, which in essence gives its last four terms.

\begin{theorem} \label{exseqgraded}
	If $L/K$ is a $\zp$-extension, then we have a canonical exact sequence
	\begin{multline*}
		0 \to {}_{I^n}(I Y_L)\{n\} 
		\to \frac{\ker \psi_{L/K}^{(n)}}{N_{L/K}(\mc{U}_L \otimes_{\zp} 
		\Omega/I^n)} 
		\to (\ker R_{L/K})\{n\}
		\\\to \frac{I^n Y_L}{I^{n+1} Y_L}
		\xrightarrow{R_{L/K}^{(n)}} \coker \psi_{L/K}^{(n)}  
		\to (\coker R_{L/K})\{n\} \to 0
	\end{multline*}
	of $\Lambda$-modules.
\end{theorem}

\begin{proof}	
	If $n = 0$, this follows by definition of $R_{L/K}$ and $\psi_{L/K}^{(0)} \colon 0 \to Y_K$, so we may assume $n \ge 1$.
	Since $H$ has $p$-cohomological dimension $1$, employing \eqref{H1HZp}, we have
	$$
		{}_{I^n} Y_L \cong \ker({}_{I^n} H^2_S(L,\zp(1))  \to {}_{I^n} H^2_{l,S}(L,\zp(1))).
	$$
	The map $J_1 \to J_{1,0}$ with $\Omega/I^n(1)$-coefficients that is identified by \eqref{H1HZp} with a map
	$$
		H^1_S(K,\Omega/I^n(1)) \to {}_{I^n} H^2_S(L,\zp(1))\{n\}
	$$
	therefore restricts to a map $\xi_{L/K}^{(n)} \colon
	\mc{U}_{L/K}^{(n)} \to {}_{I^n} Y_L\{n\}$.  
	Since $\partial_l^{(n)}$ is injective by Lemma \ref{local_bound_inj}, we also have
	$$
		{}_{I^n} Y_L \cong \ker({}_{I^n} H^2_S(L,\zp(1))  \to H^2_{l,S}(L,\zp(1))_H)
	$$
	and thereby a map of exact sequences
	$$
		 \xymatrix{
			0 \ar[r] & \mc{U}_{L/K}^{(n)} \ar[r] \ar[d]^{\xi_{L/K}^{(n)}} & H^1_S(K,\Omega/I^n(1)) \ar[r] \ar@{->>}[d] & 
			H^2_{l,S}(K,\zp(1))\{n\} \ar[d]^{\wr} \\
			0 \ar[r] & {}_{I^n} Y_L\{n\} \ar[r] & {}_{I^n} H^2_S(L,\zp(1))\{n\}
			\ar[r] & H^2_{l,S}(L,\zp(1))_H\{n\},
		}
	$$
	with the rightmost map being the inverse of the semilocal corestriction by \eqref{keysq}.
	Chasing the diagram and applying the spectral sequence $\mb{J}$ for $\Omega/I^n(1)$, 
	we conclude that $\xi_{L/K}^{(n)}$ is surjective with kernel
	$$
		\ker \xi_{L/K}^{(n)} \cong 
		N_{L/K}(\mc{U}_L \otimes_{\zp} \Omega/I^n).
	$$
	(In fact, the latter group is isomorphic to $H^1_S(L,\Omega/I^n(1))_H$ 
	by the exactness of \eqref{threeterm} for $\mb{J}$.)

	Next, we have a commutative square
	$$
		 \xymatrix@C=30pt@R=30pt{
			{}_{I^n} Y_L\{n\} \ar[r] & (Y_L)_H\{n\} \ar[d]^{R_{L/K} \otimes \mr{id}} \\
			\mc{U}_{L/K}^{(n)} \ar[r]^{\psi_{L/K}^{(n)}} \ar@{->>}[u]^{\xi_{L/K}^{(n)}} & Y_K\{n\}
		}
	$$
	where the maps arise as restrictions of the maps in \eqref{keysq}
	and hence by Lemma \ref{bocklemma} a map 
	$$
		R_{L/K}^{(n)} \colon I^n Y_L /I^{n+1} Y_L \to \coker \psi_{L/K}^{(n)}
	$$
	on cokernels of the horizontal maps.  
	Set $\mc{K}^{(n)} = N_{L/K}(\mc{U}_L \otimes_{\zp} \Omega/I^n)$.  The result follows from the map of exact sequences
	$$
		\footnotesize
		 \xymatrix{
			0  \ar[r] & {}_{I^n} (IY_L)\{n\} \ar[r] \ar[d] &
			{}_{I^n} Y_L\{n\} \ar[r] \ar[d]^{\wr} & (Y_L)_H\{n\} 
			\ar[d]^{R_{L/K} \otimes \mr{id}} \ar[r] & I^n Y_L /I^{n+1} Y_L \ar[d]^{R_{L/K}^{(n)}} \ar[r]
			& 0 \\
			0 \ar[r] & (\ker \psi_{L/K}^{(n)})/\mc{K}^{(n)} \ar[r] & \mc{U}_{L/K}^{(n)}/\mc{K}^{(n)} 
			\ar[r] & Y_K\{n\} \ar[r] & \coker \psi_{L/K}^{(n)} \ar[r] & 0,
		}
	$$
	where we have applied Lemma \ref{bocklemma} to get the upper left term.
\end{proof}

\begin{remark}
	In the proof of Lemma \ref{local_bound_inj}, 
	instead of $T = \zp(1)$, we can take $T$ to be any compact $R\ps{G_{F,S}}$-module (for any $R$ as before) such that 
	$T(-1)$ has trivial $G_{K,S}$-action.  In fact, what is actually needed is that the $H$-invariant group of $IH^2_{l,S}(L,T)$ is zero.  
	Theorem \ref{exseqgraded} then holds with $\zp(1)$ replaced by any such $T$ in the definitions and its statements.
\end{remark}

\appendix
\section{Continuous cohomology} \label{ctscoh}

In this appendix, we study the continuous cohomology of a profinite group $G$ with coefficients in topological $\La\ps{G}$-modules for a profinite ring $\La$.  We want to view such cohomology groups themselves as topological $\La$-modules.
As we have not found convenient references for the results we shall need, we list them here, leaving
the proofs to the reader.

\subsection{Topological modules over a profinite ring} \label{topmod}

Let $\La$ be a profinite ring, which is to say a compact, Hausdorff, totally disconnected topological ring with a basis of open neighborhoods of zero consisting of ideals of finite index.  
We denote the category of Hausdorff $\La$-modules with continuous $\La$-module homomorphisms by $\mc{T}_{\La}$, its full subcategory of locally compact (Hausdorff) $\La$-modules by $\mc{L}_{\La}$ and the full, abelian subcategories
of compact $\La$-modules and discrete $\La$-modules by $\C_{\La}$ and $\D_{\La}$, respectively.  The category of all $\La$-modules, with $\La$-module homomorphisms, will be denoted $\Mod_{\La}$.  We write $X \in \mc{A}$ to indicate that $X$ is an object of a category $\mc{A}$.  

We begin with the following standard facts.

\begin{lemma}
	Inverse (resp., direct) limits of objects in the category $\mc{T}_{\La}$ exist, and they
	are endowed with the initial (resp., final) topology with respect to the inverse (resp.,
	direct) system defining the limit.
\end{lemma}

For a proof of the following, see \cite[Lemma 5.1.1]{rz}.  

\begin{lemma} \label{limfin}
	Every finite object in $\T_{\La}$ is discrete.
	Every object in $\C_{\La}$ is an inverse limit of finite $\La$-module quotients, and
	every object in $\D_{\La}$ is a direct limit of finite $\La$-submodules.
\end{lemma}

For topological spaces $X$ and $Y$, we use
$\Maps(X,Y)$ to denote the set of continuous maps from $X$ to $Y$, which we endow
with the compact-open topology.  This topology has a subbase of open sets
$$
	V(K,U) = \{ f \in \Maps(X,Y) \mid f(K) \subseteq U \}
$$
for $K \subseteq X$ compact and $U \subseteq Y$ open.
We have the following lemma (cf.\ \cite[Proposition 3]{flood}).

\begin{lemma} \label{mapstop}
	For $X$ a topological space and
	$N \in \mc{T}_{\La}$, the abelian group $\Maps(X,N)$ is a topological 
	left $\La$-module under the action
	defined by $(\lambda \cdot f)(x) = \lambda \cdot f(x)$ for $\lambda \in \La$, 
	$f \in \Maps(X,N)$, and $x \in X$. 
\end{lemma}

\begin{lemma} \label{mapscomp}
	Let $X$ be a compact space.
	Let $(N_{\alpha} , \pi_{\alpha,\beta})$ be an inverse system of finite objects 
	and surjective maps in 
	$\C_{\La}$, and let $(N, \pi_{\alpha})$ be the inverse limit of the system.
	Then the isomorphism 
	$$
		\phi \colon \Maps(X,N) \to \invlim{\alpha} \Maps(X,N_{\alpha})
	$$
	of $\Lambda$-modules induced by the inverse limit is also a homeomorphism.
\end{lemma}

\begin{lemma} \label{mapsdisc}
	Let $(X_{\alpha},\iota_{\alpha,\beta})$ be a direct system of finite objects 
	and injective maps
	in the category of discrete spaces, and let 
	$(X,\iota_{\alpha})$ be the direct limit of the system.  
	Let $N \in \T_{\La}$. Then the isomorphism
	$$
		\theta \colon \Maps(X,N) \to \invlim{\alpha} \Maps(X_{\alpha},N)
	$$
	of $\Lambda$-modules is a homeomorphism.
\end{lemma}

Fix a set $\mc{I}$ of open ideals of $\La$ that forms a basis of neighborhoods of $0$ in $\La$.
We recall the following lemma for convenience: see \cite[Proposition 3.1.7]{lim} and \cite[Lemma 3.1.4]{lim}.

\begin{lemma}\ \label{modfacts}
	\begin{enumerate}
		\item[a.] Let $M$ be a finitely generated, compact $\La$-module.  Then we have
		an isomorphism
		$$
			M \cong \invlim{I \in \mc{I}} M/IM
		$$
		of topological $\La$-modules.
		\item[b.] Let $M$ be a finitely generated, compact $\La$-module and 
		$N$ be a compact or discrete $\Lambda$-module.  Or, let $M$
		and $N$ be compact $\La$-modules endowed with the $\mc{I}$-adic
		topology.
		Then
		$$
			\Hom_{\La,\cts}(M,N) = \Hom_{\La}(M,N).
		$$
	\end{enumerate}
\end{lemma}

The following lemma is a consequence of Lemma \ref{mapscomp}.

\begin{lemma} \label{homcomp}
	Let $M, N \in \C_{\La}$, 
	and suppose that $N$ is an inverse limit 
	in $\C_{\La}$ of a system $(N_{\alpha},\pi_{\alpha,\beta})$ of
	finite $\Lambda$-modules and surjective homomorphisms.
	Then the isomorphism 
	$$
		\Hom_{\La,\cts}(M,N) \xrightarrow{\sim} \invlim{\alpha} 
		\Hom_{\La,\cts}(M,N_{\alpha})
	$$
	of groups induced by the inverse limit is also a homeomorphism.
\end{lemma}

We also have the analogue for discrete $\La$-modules,
which is a corollary of Lemma \ref{mapsdisc}.

\begin{lemma} \label{homdisc}
	Let $M, N \in \D_{\La}$, 
	and suppose that $M$ is a direct limit 
	in $\D_{\La}$ of a system $(M_{\alpha},\iota_{\alpha,\beta})$ of
	finite $\Lambda$-modules and injective homomorphisms.
	Then the isomorphism 
	$$
		\Hom_{\La}(M,N) \xrightarrow{\sim} \invlim{\alpha} 
		\Hom_{\La}(M_{\alpha},N)
	$$
	of groups induced by the inverse limit is also a homeomorphism.
\end{lemma}

We let $\Lo$ denote the opposite ring to $\La$.  For a topological $\La$-module $X$,
we let
$$
	X^{\vee} = \Hom_{\cts}(X,\R/\Z)
$$
denote the Pontryagin dual $\Lo$-module.  
We endow $X^{\vee}$ with the compact-open topology, and then $X^{\vee}$ becomes
an object in $\mc{T}_{\Lo}$.
The Pontryagin dual preserves the property of local compactness.  For any $X \in \mc{L}_{\La}$, the natural map $X \to (X^{\vee})^{\vee}$ is an isomorphism in $\mc{L}_{\La}$.  If $X$ is a
compact (resp., discrete) $\La$-module, then $X^{\vee}$ is a discrete (resp., compact)
$\Lo$-module.  In fact, the Pontryagin dual provides contravariant, exact equivalences between $\C_{\La}$ and $\mc{D}_{\Lo}$.

One may prove the following by employing Lemmas \ref{modfacts}, \ref{homcomp}, and \ref{homdisc}.

\begin{proposition} \label{dualhom}
	Suppose that $\La$ is an $R$-algebra over a profinite commutative ring $R$ via a continuous map from $R$ to the 
	center of $\La$.
	Suppose that $T, U \in \C_{\La}$ are endowed with the $\mc{I}$-adic
	topology (e.g., are finitely generated $\La$-modules).
	Then the map 
	$$
		\Hom_{\La}(T,U) \to \Hom_{\Lo}(U^{\vee},T^{\vee})
	$$
	that sends $\rho$ to the map $\rho^*$ with
	$\rho^*(\phi) = \phi \circ \rho$ 
	for all $\phi \in U^{\vee}$ is an isomorphism in $\C_R$.
\end{proposition}

\subsection{Cochains and cohomology groups} \label{cochcoh}

Let $G$ be a profinite group.
We set $\mc{T}_{\La,G} = \mc{T}_{\La\ps{G}}$, which is to say that $\mc{T}_{\La,G}$ is the category of topological $\La$-modules with a continuous $\La$-linear action of $G$ with morphisms that are continuous $\La[G]$-module homomorphisms.  We similarly define $\mc{L}_{\La,G}$, $\C_{\La,G}$, and $\D_{\La,G}$.

The category of chain complexes over an additive category $\mc{A}$ that admits kernels and cokernels will be denoted by $\Ch(\mc{A})$, with $\Ch^+(\mc{A})$, $\Ch^-(\mc{A})$, and $\Ch^b(\mc{A})$ denoting the full subcategories of bounded below, bounded above, and bounded complexes in $\mc{A}$.

The complex of inhomogeneous continuous cochains provides a functor 
$$
	C(G, \,\cdot \,) \colon \mc{T}_{\La,G} \to \Ch^+(\Mod_{\La}).
$$
For $M \in \T_{\La,G}$, we endow
each $C^i(G,M)$ with the compact-open topology.
The coboundary maps in the complex $C(G,M)$ 
are then continuous as $G$ is 
a topological group, $M$ is a continuous $G$-module, and the action of $\La$ on 
$M$ commutes with the $G$-action.  For $M \in \mc{T}_{\La,G}$, we denote the $\La$-modules 
that are the $i$th cocycle, coboundary, and cohomology groups for $C(G, M)$ by $Z^i(G,M)$, 
$B^i(G,M)$, and $H^i(G,M)$, respectively.
Both $Z^i(G,M)$ and $B^i(G,M)$ may be viewed as objects in $\T_{\La}$ under the 
subspace topologies from $C^i(G,M)$.  

\begin{lemma} \label{cochdisc}
	For $A \in \mc{D}_{\La,G}$, the topological $\La$-modules $C^i(G,A)$ are discrete.
\end{lemma}

If $A \in \D_{\La,G}$, then $A$ is a direct limit of its finite $\La[G]$-submodules $A_{\alpha}$, and we have
$$
	H^i(G,A) \cong \dirlim{\alpha} H^i(G,A_{\alpha}).
$$ 
We endow $H^i(G,A)$ with the discrete topology, under which it is an object of $\D_{\La}$.
This of course agrees with the subquotient topology from $C^i(G,A)$, where we view $H^i(G,A)$
as $Z^i(G,A)/B^i(G,A)$.   
The result in the compact case
 is found in the following statement, the proof of which employs Lemma \ref{mapscomp}.

\begin{proposition} \label{topctscoh}
	Assume that $H^i(G,M)$ is finite for every finite $\Z[P^{-1}][G]$-module 
	$M$ and every $i \ge 0$,
	where $P$ is a set of primes of $\Z$ that act invertibly on $\Lambda$.
	Suppose that $T \in \C_{\La,G}$, and write $T = \invlim{} T_{\alpha}$ for
	some finite $\La[G]$-module quotients $T_{\alpha}$.  Then the natural map
	$$
		H^i(G,T) \to \invlim{\alpha} H^i(G,T_{\alpha})
	$$
	is an isomorphism of $\La$-modules. 
	Moreover, $B^i(G,T)$ and $Z^i(G,T)$ are closed subspaces of $C^i(G,T)$, and 
	the subquotient topology on $H^i(G,T)$ induced by the isomorphism 
	$$
		H^i(G,T) \cong \frac{Z^i(G,T)}{B^i(G,T)}
	$$ 
	agrees with the profinite topology induced by the above isomorphism.
\end{proposition}

Cup products on continuous cohomology exist quite generally, as stated in the following lemma.  
For this, if $\Omega$ denotes a profinite
ring, then  $\T_{\Omega-\La,G}$ denotes the category of topological 
$\Omega$-$\La$-bimodules with a continuous commuting action of $G$.

\begin{lemma} \label{cupprod}
	Let $\La$, $\Omega$, and $\Sigma$ be profinite rings.  
	Let $M \in \T_{\Omega-\La,G}$, $N \in \T_{\La-\Sigma,G}$, and 
	$L \in \T_{\Omega-\Sigma,G}$, and suppose that 
	$\phi \colon M \times N \to L$
	is a continuous, $\La$-balanced, 
	$G$-equivariant homomorphism of $\Omega$-$\Sigma$-bimodules. 
	Then we have continuous, $\La$-balanced cup products 
	$$
		C^i(G,M) \times C^j(G,N) \xrightarrow{\cup} C^{i+j}(G,L)
	$$
	of $\Omega$-$\Sigma$-bimodules for each $i, j \ge 0$.
\end{lemma}

\section{A comparison of Poitou-Tate and Kummer maps}

\numberwithin{theorem}{section}

Let $p$ be a prime number.  Let $F$ be a number field, and let $S$ denote a finite set of primes of $F$ including those above $p$ and any real places.   We assume that $p$ is odd or $F$ is purely imaginary.  

Poitou-Tate duality provides us with a canonically defined homomorphism
$$
	H^1(G_{F,S},\qp/\zp)^{\vee} \to H^2_{\cts}(G_{F,S},\zp(1)),
$$
and hence, via our identifications, a homomorphism $\mf{X}_F \to Y_F$.  The natural question to ask is whether or not this map is the restriction map on Galois groups, and the answer is that in fact it is.  However, at the time of the writing of this appendix as a note, we were unable to find a proof of this nonobvious but useful fact in the literature.  This fact is not strictly necessary in the paper itself, but we feel that this should appear somewhere in print and fits very nicely with the theme of the present article.  We thank Kay Wingberg and Alexander Schmidt for helpful discussions regarding the proof.

\begin{theorem} \label{mapisnat}
	The Poitou-Tate map $H^1(G_{F,S},\qp/\zp)^{\vee} 
	\to H^2_{\cts}(G_{F,S},\zp(1))$ induces the restriction map
	$\mf{X}_F \to Y_F$ on Galois groups.
\end{theorem}

\begin{proof}
	Let $F_S$ denote the Galois group of the maximal unramified outside $S$-extension of $F$,
	and let $\mc{O}_S$ denote its ring of $S$-integers.
	Moreover, let $I_S$ and $C_S$ be the $S$-id\`ele group and $S$-id\`ele class group of $F_S$, respectively,
	and let $I_S(F)$ and $C_S(F)$ denote the $S$-id\'ele group and $S$-id\'ele class group of $F$, 
	respectively.
	Finally, let $\mr{Cl}_{F,S}$ denote the $S$-class group of $F$.

	We first recall the definition of the Poitou-Tate map
	$$
		H^1(G_{F,S},\qp/\zp)^{\vee} \to H^2_{\cts}(G_{F,S},\zp(1)).
	$$
	We find it most convenient to work modulo $p^n$ throughout and then take inverse limits.
	Modulo $p^n$, the Poitou-Tate map arises simply as the composition
	$$
		H^1(G_{F,S},\Z/p^n\Z)^{\vee} \to H^0(G_{F,S},\Hom(\mu_{p^n},C_S))^{\vee} \to H^2(G_{F,S},\mu_{p^n}),
	$$
	where the first map is the dual of the connecting homomorphism arising from the sequence
	$$
		0 \to \mu_{p^n} \to \Hom(\mu_{p^n},I_S) \to \Hom(\mu_{p^n},C_S) \to 0
	$$
	and the second map arises from the duality
	$$
		H^0(G_{F,S},\Hom(\mu_{p^n},C_S)) \times H^2(G_{F,S},\mu_{p^n}) \xrightarrow{\cup} H^2(F,C_S) 
		\xrightarrow{\mr{inv}} \Q/\Z,
	$$
	where ``$\mr{inv}$" denotes the invariant map.
	
	Next, we explain the injection $Y_F \hookrightarrow 
	H^2_{\cts}(G_{F,S},\zp(1))$, again working modulo $p^n$.
	Recall that the reciprocity homomorphism
	$$
		\mr{rec} \colon C_S(F)/p^n \to \mf{X}_F/p^n
	$$
	is dual to the connecting homomorphism 
	$$
		\delta \colon H^1(G_{F,S},\Z/p^n\Z) \to H^2(G_{F,S},\Z)[p^n]
	$$
	under the cup product
	$$
		H^2(G_{F,S},\Z)[p^n] \times H^0(G_{F,S},C_S)/p^n \xrightarrow{\cup} 
		H^2(G_{F,S},C_S)[p^n] \xrightarrow{\mr{inv}} \Z/p^n\Z
	$$
	in the sense that 
	$$
		\mr{inv} (\delta\phi \cup a) = \phi(\mr{rec}(a))
	$$
	for $\phi \in H^1(G_{F,S},\Z/p^n\Z)$ and $a \in C_S(F)/p^n$.
	
	From the long exact sequence attached to
	$$
		0 \to \mc{O}_S^{\times} \to I_S \to C_S \to 0
	$$
	and the fact that the cokernel of $I_S(F) \to C_S(F)$ is isomorphic to $\mr{Cl}_{F,S}$, 
	we have an isomorphism $H^1(G_{F,S},\mc{O}_S^{\times}) \cong \mr{Cl}_{F,S}$, and an induced
	reciprocity map
	$$
		\mr{rec} \colon \mr{Cl}_{F,S}/p^n \to Y_F/p^n.
	$$
	From the long exact sequence attached to
	$$
		0 \to \mu_{p^n} \to \mc{O}_S^{\times} \xrightarrow{p^n} \mc{O}_S^{\times} \to 0,
	$$
	we obtain an injection that is the composite
	$$
		Y_F/p^n \xrightarrow{\mr{rec}^{-1}}  \mr{Cl}_{F,S}/p^n \xrightarrow{\sim}
		H^1(G_{F,S},\mc{O}_S^{\times})/p^n \hookrightarrow H^2(G_{F,S},\mu_{p^n}).
	$$
	This gives the identification of $Y_F/p^n$ with a subgroup of $H^2(G_{F,S},\mu_{p^n})$
	arising from Kummer theory and class field theory.
	
	Putting all of the definitions together, the proposition is reduced to the commutativity of the diagram
	$$
		 \xymatrix@C=30pt{ 
		\Hom(\mu_{p^n},C_S)^{G_{F,S}} \times  H^2(G_{F,S},\mu_{p^n})  \ar@<-7ex>[d]  \ar[r]^-{\cup} & \Q/\Z \\
		H^1(G_{F,S},\Z/p^n\Z) \times  H^1(G_{F,S},\mc{O}_S^{\times}) \ar@<-7ex>[d]  \ar@<-8ex>[u]  \\
		\quad H^2(G_{F,S},\Z) \ \ \ \times  H^0(G_{F,S},C_S) \ar@<-8ex>[u] \ar[r]^-{\cup} & \Q/\Z \ar@{=}[uu], 
		}
	$$
	in the obvious sense.  This diagram can be found (without proof of its commutativity) in 
	\cite[(8.4.6)]{nsw0}, though not in the second edition of the book.
	
	To show its commutativity, we replace the
	right-hand composition with the composition
	$$
		H^0(G_{F,S},C_S) \to H^1(G_{F,S},C_S[p^n]) \to H^2(G_{F,S},\mu_{p^n})
	$$
	which is in fact its negative by a standard lemma (e.g., \cite[(1.3.4)]{nsw0}), 
	since we have a commutative diagram
	$$
		 \xymatrix{
		& 0 \ar[d] & 0 \ar[d] & 0 \ar[d] & \\
		0 \ar[r] & \mu_{p^n} \ar[r] \ar[d] & I_S[p^n] \ar[r] \ar[d] & C_S[p^n] \ar[r] \ar[d] & 0 \\
		0 \ar[r] & \mc{O}_S^{\times} \ar[r] \ar[d]^{p^n} & I_S \ar[r] \ar[d]^{p^n} & C_S \ar[r] \ar[d]^{p^n} & 0 \\
		0 \ar[r] & \mc{O}_S^{\times} \ar[r] \ar[d] & I_S \ar[r] \ar[d] & C_S \ar[r] \ar[d] & 0 \\
		& 0 & 0 & 0 & 
		}
	$$
	noting the $p$-divisibility of $C_S$ \cite[(10.9.5)]{nsw0}.  We therefore have a new diagram
	$$
		 \xymatrix@C=30pt{ 
		\Hom(\mu_{p^n},C_S)^{G_{F,S}}  \times \ \,  H^2(G_{F,S},\mu_{p^n}) \ar@<-7ex>[d] \ar[r]^-{\cup} & \Q/\Z \\
		\ \ H^1(G_{F,S},\Z/p^n\Z) \  \times H^1(G_{F,S},C_S[p^n]) \ar@<-9ex>[u]  \ar@<-7ex>[d]  \ar[r]^-{\cup} & \Q/\Z \ar[u]_{-1} \\
		\quad H^2(G_{F,S},\Z) \ \ \ \,    \times \ \    H^0(G_{F,S},C_S) \ar@<-9ex>[u] \ar[r]^-{\cup} & \Q/\Z \ar@{=}[u], }
	$$
	which commutes by two applications of \cite[Theorem III.2.1]{lang}.  That is, for the lower rectangle,
	take the pairing $\Z \times C_S \to C_S$, and for the upper, take the Galois-equivariant pairing
	$$
		I_S[p^n] \times \Hom(\mu_{p^n},I_S[p^n]) \to C_S[p^n]
	$$
	given by multiplication on $I_S[p^n]$ followed by projection, noting that
	$$
		I_S[p^n] = \lim_{\substack{\rightarrow \\ E \subset F_S}} \bigoplus_{v \in S_E} \mu_{p^n}(E_v).
	$$
	The result follows.
\end{proof}


\begin{thebibliography}{WWWW}
\bibitem[Fl]{flood} J. Flood, Pontryagin duality for topological modules, \textit{Proc. Amer.
Math. Soc.} {\bf 75} (1979), 329--333.
\bibitem[FK1]{fk} T. Fukaya and K. Kato, A formulation of conjectures on $p$-adic zeta functions in noncommutative Iwasawa theory, Proceedings of the St. Petersburg Mathematical Society, Vol. XII, 1--85, {\em Amer. Math. Soc. Transl. Ser. 2} {\bf 219}, Amer. Math. Soc., Providence, 2006.
\bibitem[FK2]{fk-proof} T. Fukaya and K. Kato, On conjectures of Sharifi, preprint, 121 pages, 2012.
\bibitem[GV]{gv} R. Greenberg and N. Vatsal, On the Iwasawa invariants of elliptic curves,
\textit{Invent. Math.} {\bf 142} (2000), 17--63.
\bibitem[HS]{hs} Y. Hachimori and R. Sharifi, 
  On the failure of pseudo-nullity of Iwasawa modules,
  \textit{J. Alg. Geom.} \textbf{14} (2005), 567--591.
 \bibitem[Iw2]{iwasawa} K. Iwasawa, On cohomology groups of units for
    $\zp$-extensions, \textit{Amer. J. Math.} \textbf{105} (1983), 189--200.
 \bibitem[La]{lang} S. Lang, Topics in the cohomology of groups, Lecture Notes in Mathematics, vol. 1625,
	Springer, Berlin, 1996.
\bibitem[LLSWW]{llsww} J. Lam, Y. Liu, R. Sharifi, J. Wang, and P. Wake, Generalized Bockstein maps and Massey products, 
	preprint, arXiv:2004.11510.
\bibitem[Li]{lim} M. F. Lim, Poitou-Tate duality over extensions of global fields, \textit{J. Number Theory}
{\bf 132} (2012), 2636--2672.
\bibitem[LS]{ls} M. F. Lim and R. Sharifi, Nekovar duality over $p$-adic Lie extensions
  of global fields, \textit{Doc. Math.} \textbf{18} (2013), 621--678.
\bibitem[McS]{mcs} W. McCallum and R. Sharifi, A cup product in
  the Galois cohomology of number fields, \textit{Duke Math. J.} {\bf 120} (2003),
  269--310.
\bibitem[Ne]{nekovar} J. Nekovar, Selmer complexes, \textit{Asterisque} {\bf 310} (2006).
\bibitem[NSW1]{nsw0} J. Neurkich, A. Schmidt, and K. Wingberg, 
	Cohomology of Number Fields, \textit{Grundlehren Math. Wiss.} \textbf{323}, Springer-Verlag, Berlin, 2000.
\bibitem[NSW2]{nsw} J. Neukirch, A. Schmidt and K. Wingberg,
Cohomology of Number Fields, Second Edition, \textit{Grundlehren
Math. Wiss.} \textbf{323}, Springer-Verlag, Berlin, 2008.
\bibitem[RZ]{rz} L. Ribes and P. Zalesskii, Profinite Groups,
Second edition, \textit{Ergeb. Math. Grenzgeb.} (3) \textbf{40}, Springer-Verlag, Berlin, 2010.
\bibitem[Sh1]{me-massey} R. Sharifi, Massey products and ideal   class groups, \textit{J. reine angew. Math.} \textbf{603} (2007), 1--33.
\bibitem[Sh2]{me-gal} R. Sharifi, On Galois groups of unramified pro-$p$ extensions, \textit{Math. Ann.} {\bf 342} (2008), 297--308.
\bibitem[Sh3]{me-Lfn} R. Sharifi, A reciprocity map and the two variable $p$-adic $L$-function, \textit{Ann. of Math.} {\bf 173} (2011), 251--300.
\bibitem[Sh4]{me-selmer} R. Sharifi, Selmer groups of residually reducible representations, in preparation.
\end{thebibliography}
\end{document}